\documentclass[12pt]{article}
\usepackage{epsf}
\usepackage{amsbsy,amsmath}
\usepackage{mathtools}
\usepackage{mathrsfs}
\usepackage{amsfonts}
\usepackage{amssymb}
\usepackage{enumitem}
\usepackage{eucal}
\usepackage{enumitem}
\usepackage{graphics,mathrsfs}
\usepackage{amsthm}
\usepackage{secdot}
\usepackage{xcolor,soul}

\newtheorem{theorem}{Theorem}[section]
\newtheorem{lemma}[theorem]{Lemma}

\theoremstyle{definition}
\newtheorem{definition}[theorem]{Definition}

\theoremstyle{remark}
\newtheorem{remark}[theorem]{Remark}
\numberwithin{equation}{section}

\numberwithin{equation}{section}

\newsavebox{\savepar}

\begin{document}
	
	\title{Existence and H\"{o}lder regularity of infinitely many solutions to a $p$-Kirchhoff type problem involving a singular nonlinearity without the Ambrosetti-Rabinowitz (AR) condition}
	\author{Debajyoti Choudhuri$^\ddagger$ \\
		\small{$^\ddagger$Department of Mathematics, National Institute of Technology Rourkela, India}\\
		\small{Emails: dc.iit12@gmail.com}}
	\date{}
	\maketitle
\begin{abstract}
\noindent We carry out an investigation of the existence of infinitely many solutions to a fractional $p$-Kirchhoff type problem with a singularity and a superlinear nonlinearity with a homogeneous Dirichlet boundary condition. Further the solution(s) will be proved to be bounded and a weak comparison principle has also been proved. A {\it `$C^1$ versus $W_0^{s,p}$'} analysis has also been discussed.
\begin{flushleft}
\end{flushleft}
\end{abstract}
\section{Introduction}
Off late, the problems involving a nonlocal and fractional operators have become hugely popular area of investigation owing to its manifold applications, viz. stratified materials, population dynamics, continuum mechanics, water waves, minimal surface problems etc. Interested readers may refer to \cite{fisval,ian,barr,moli1,moli2,tang,serva1} and the references therein.\\
The problem addressed in this article is as follows.
{\small\begin{align}\label{main}
\left(a+b\int_{\mathbb{R}^{2N}}|u(x)-u(y)|^{p}K(x-y)dxdy\right)\mathfrak{L}_p^su-\lambda g(x) |u|^{p-2}u=&\mu \frac{h(x)u}{|u|^{\gamma+1}}+f(x,u),~\text{in}~\Omega\nonumber\\
u=&0,~\text{in}~\mathbb{R}^N\setminus\Omega
\end{align}}
where $\lambda,\mu>0$, {$g,h\geq 0$} are functions that are defined and bounded over $\Omega$. Here $\Omega\subset\mathbb{R}^N$ is a bounded domain with Lipschitz boundary $\partial\Omega$ for $N\geq 2$, $a>0,b\geq0$, $0<\gamma,s<1<p<\infty$, $sp<N$. The function $f$ is a {Carath\'{e}odory} function and the operator $\mathfrak{L}_p^s$ is defined as 
$$\mathfrak{L}_p^su=2\underset{\epsilon\rightarrow 0^+}{\lim}\int_{\mathbb{R}^N\setminus B_{\epsilon}(x)}|u(x)-u(y)|^{p-2}(u(x)-u(y))K(x-y)dy$$
for all $x\in\mathbb{R}^N$ where {$B_{\epsilon}(x)=\{y:|y-x|<\epsilon\}$}. The function $K:\mathbb{R}^N\setminus\{0\}\rightarrow(0,\infty)$ is measurable with the following properties:
\begin{align}\label{kernel}
\begin{split}
&(i)~\mathfrak{\rho}K\in L^1(\mathbb{R}^N)~\text{where}~\mathfrak{\rho}(x)=\min\{|x|^2,1\} \\
&(ii)~\text{There exists}~ \delta>0~\text{such that}~K(x)\geq \delta|x|^{-N-ps}~\text{for all}~x\in\mathbb{R}^N\\
&(iii)~K(x)=K(-x)~\text{for all}~x\in\mathbb{R}^N\setminus\{0\}.
\end{split}
\end{align}
One can retrieve the fractional $p$-Laplace operator if the {\it `kernel'} $K$ is chosen to be $K(x)=|x-y|^{-N-ps}$. The discussion in Section 4.1 uses following condition $(P)$.
$$(P'):~\delta_1|x|^{-N-ps}\geq K(x)\geq \delta_2|x|^{-N-ps}$$ for all $x\in\mathbb{R}^N$ where $\delta_1,\delta_2>0$. In general, it is a practice to denote the Kirchhoff function as $\mathfrak{M}$. In the current case $\mathfrak{M}(t)=a+bt^p$ where $a,b \geq 0$ with $a+b>0$. Such a Kirchhoff function is termed as the nondegenerate ones. Therefore when $\mathfrak{M}(t)\equiv 1$, $\lambda=0$, $p=2$, $g(x)=1$, {$h(x)=1$} a.e. in $\Omega$, we reduce to the problem in \eqref{main} to
\begin{align}\label{submain}
\begin{split}
(-\Delta)^su&=\mu u^{-\gamma}+f(x,u),~\text{in}~\Omega\\
u&>0,~\text{in}~\Omega \\
u&=0,~\text{in}~\mathbb{R}^N\setminus\Omega. 
\end{split}
\end{align}
For further details on the problem in \eqref{submain} one may refer to \cite{serva1}. The problem in \eqref{submain} with $\mu=0$ has been addressed in \cite{serva2}. The authors have used a variational technique to guarantee the existence of multiple solutions. Further  results on existence of multiple solutions can be found in \cite{ghanmi2016nehari, mukherjee2016dirichlet}. In most of these studies, the authors obtained two distinct weak solutions. Readers may refer to \cite{saoudi_sugg_3, saoudi_sugg_1, saoudisugg2, saoudi_sugg_4, ming3, xiang4} for ideas and techniques developed in order to guarantee the existence of infinitely many solutions to nonlocal elliptic problems driven by a singularity.\\
Meanwhile, we direct the readers to a variety of forms for the function $\mathfrak{M}$ \cite{bin,alv,moli3, zuo2, zuo1, ming1,ming2, xiang1,xiang2}. With the advent of these references and with the help of fountain theorem, the authors in \cite{zuo3} have proved the existence of infinitely many solutions for a fractional $p$-Kirchhoff problem. A use of the fountain theorem has also been made in \cite{xiang3} to guarantee the existence of infintely many solutions for a fractional Kirchhoff type problem. In \cite{nya}, the authors showed the existence and multiplicity of solutions to a degenerate fractional $p$-Kirchhoff problem. Fiscella et al in \cite{fis1} dealt with Kirchhoff type equation over $\mathbb{R}^N$ involving Hardy-Sobolev nonlinearities with critical exponent. We also refer to \cite{zuo4, cap1, pucci} for a related problem. Recently, Ghosh \cite{ghosh1} has proved the existence of infinitely many solutions to a system of fractional Laplacian Kirchhoff type problem with a sublinear growth. For articles that addressed the problem of $p$-Kirchhoff system, one may refer to \cite{mish} and the references therein. A futuristic research in this direction is to consider the parabolic counterparts to these problems. The readers may refer to the work due to \cite{pan1} and the references therein. Motivated from the work due to Ren et al \cite{ren1} we will show the existence of infinitely many solutions to a $p$-Kirchhoff type problem with a superlinear growth without the AR condition. It will also be proved that the solution (if exists) is in $L^{\infty}(\Omega)$. A weak comparison principle has also been proved. A little bit of history about the AR condition  - this condition was first introduced by Ambrosetti and Rabinowitz (refer \cite{amb1}) in 1973. Thereafter this condition formed a formidable tool in the analysis of elliptic PDEs, especially to prove the boundedness of the Palais-Smale (PS) sequences for the associated energy functional to the problem. To our knowledge there is no evidence in the literature that considered a nonlocal $p$-Kirchhoff type problem with a singular nonlinearity without the AR condition. Therefore the problem considered and the results obtained here are new.
\section{A simple physical motivation}
This section is devoted to a physical motivation to the problem considered in this article. The explanation is physically heuristic but nevertheless gives a strong mathematical motivation to take up this problem. {We will restrict to the one dimensional case of the model of an elastic string of finite length fixed at both the ends.} The vertical displacement of the string will be represented by $u:[-1,1]\times[0,\infty)\rightarrow\mathbb{R}$. Then, mathematically, the end point constraints can be expressed as 
$$u(0,t)=u(2,t)=0$$ for all $t\geq 0$. In order to identify this finite string with an infinite string one can consider 
$$u(x,t)=0$$ for all $x\in\mathbb{R}\setminus[0,2]$, $t\geq 0$.
Thus, the acceleration {$\frac{\partial^2u}{\partial t^2}$} of the vertical displacement $u$ of the vibrating string must be balanced (thanks to Newton's laws) by the elastic force of the string and by the external force field $f$. Therefore we have
$$\frac{\partial^2u}{\partial t^2}=m\frac{\partial^2u}{\partial x^2}+f,~\text{for all}~x~\text{in}~[0,2],~ t\geq 0.$$
When the steady case is considered, we have 
$$ m\frac{\partial^2u(x)}{\partial x^2}=f(x),~\text{in}~[0,2].$$
We quote here G.F. Carrier \cite{carrier}  - {\it it is well known that the classical linearized analysis of the vibrating string can lead to results which are reasonably accurate only when the minimum (rest position) tension and the displacements are of such magnitude that the relative change in tension during the motion is small}. Taking this into account one can suppose that the tension due to small deformation is linear in form, then we have the following expression.
$$\mathfrak{M}(l)=m_0+2Cl$$
where $l$ is the increment in the length of the string with respect to its mean position, i.e.
$$l=\int_{0}^{2}\sqrt{1+\left(\frac{\partial u}{\partial x}\right)^2}dx-2,$$ $C>0$ is a constant of proportionality. 
Thus for small deformations we have 
$$\sqrt{1+\left(\frac{\partial u}{\partial x}\right)^2}=1+\frac{1}{2}\left(\frac{\partial u}{\partial x}\right)^2.$$
Hence,
$$l=\frac{1}{2}\left(\frac{\partial u}{\partial x}\right)^2.$$
 Therefore, the problem now boils down to a Kirchhoff type problem
\begin{eqnarray}
 \left(m_0+C\frac{1}{2}\left(\frac{\partial u}{\partial x}\right)^2\right)\frac{\partial^2u(x)}{\partial x^2}&=&f(x);~\text{for all}~x~\text{in}~[0,2],\nonumber\\
u(x)&=&0,~\text{in}~\mathbb{R}\setminus[0,2].
\end{eqnarray}
In other words it models the vibration of a string.  

\section{Technical preliminaries and functional analytic set up}
We begin by giving the conditions of the function $f$ which is assumed to have a superlinear growth. We quickly will recall the AR condition, which is as follows.
\begin{align}\label{ARcond}
(AR):~&\text{there exist constants}~r>0, 1<\eta<\theta~\text{such that}\nonumber\\
&0<\theta F(x,t)\leq tf(x,t)~\text{for any}~x\in\Omega, t\in\mathbb{R}~\text{and}~|t|\geq r.
\end{align}
Here $F(x,t)=\int_{0}^{t}f(x,t)dt$. From \eqref{ARcond} we have that $F(x,t)\geq c_1|t|^{\theta}-c_2$, for any $(x,t)\in\Omega\times\mathbb{R}$, where $\theta>\eta$ for $c_1, c_2>0$ constants. However, there are still many functions which are superlinear at infinity yet not satisfying the AR condition. We make note of another form which is given by 
\begin{eqnarray}\label{ARanotherform}
\lim_{|t|\rightarrow\infty}\frac{F(x,t)}{|t|^{\eta}}=\infty~\text{uniformly for}~x\in\Omega.
\end{eqnarray} 
{A closer observation shows that the nonlinearity $f$, under the condition} \eqref{ARanotherform}, {is also superlinear at infinity. Clearly, for example, the function $f(x,t)=t^{\eta-1}\log(1+t)$ satisfies} \eqref{ARanotherform} {but does not satisfy} \eqref{ARcond}. Some important works that has proved the existence of infinitely many solutions for nonlocal equations of fractional Laplacian type driven by a superlinear term and with homogeneous Dirichlet boundary data, but without the AR condition can be found in \cite{thin1}. We now give the assumptions made on the function $f:\Omega\times\mathbb{R}\rightarrow\mathbb{R}$.
\begin{eqnarray}\label{f_cond}
&(f_1)& \exists C>0~\text{and}~q\in(p,p_s^*)~\text{such that}~|f(x,t)|\leq C(1+|t|^{q-1})\nonumber\\
&(f_2)& f(x,-t)=-f(x,t) \forall (x,t)\in\Omega\times\mathbb{R}~\text{and}~f\in C(\Omega\times\mathbb{R},\mathbb{R}) \nonumber\\
&(f_3)& \lim_{|t|\rightarrow\infty}\frac{F(x,t)}{|t|^{2p}}=\infty~\text{uniformly for all}~x\in\Omega\nonumber\\
&(f_4)& \lim_{|t|\rightarrow 0}\frac{f(x,t)}{|t|^{p-1}}=0~\text{uniformly for all}~x\in\Omega\nonumber\\
&(f_5)& \exists \overline{t}>0~\text{such that}~t\mapsto \frac{f(x,t)}{t^{2p-1}}~\text{is decreasing if}~t\leq-\overline{t}<0\nonumber\\
& &~\text{and increasing if}~t\geq\overline{t}>0\forall x\in\Omega\nonumber\\
&(f_6)& \exists \sigma\geq 1~\text{and}~T\in L^1(\Omega)~\text{satisfying}~T(x)\geq 0~\text{such that}~\mathfrak{G}(x,s)\leq\sigma\mathfrak{G}(x,t)+T(x)\nonumber\\
& &\forall x\in\Omega~\text{and}~0\leq|s|\leq|t|,~\text{where}~\mathfrak{G}(x,t)=\frac{1}{2p}tf(x,t)-F(x,t).\nonumber
\end{eqnarray}
\begin{remark}\label{condprop}
The condition $(f_6)$ was assumed by Jeanjean \cite{jean}. When $\sigma=1$ one can see that the conditions $(f_5)$ and $(f_6)$ are equivalent. In general, there are functions (for example $f(x,t)=2p|t|^{2p-2}t\ln(1+t^{2p})+p\sin t$) that satisfy $(f_6)$ but not $(f_5)$.
\end{remark}
\noindent We assign $Q=\mathbb{R}^{2N}\setminus C(\Omega)\times C(\Omega)$ where $C(\Omega)=\mathbb{R}^N\setminus\Omega${. The} space $X$ will denote the space of Lebesgue measurable functions from $\mathbb{R}^N$ to $\mathbb{R}$ such that its restriction to $\Omega$ of any function $u$ in $X$ belongs to $L^p(\Omega)$ and 
$$\int_{Q}|u(x)-u(y)|^{p}K(x-y)dxdy<\infty.$$
The space $X$ is equipped with the norm 
\begin{eqnarray}
\|u\|_X&=&\|u\|_{L^p(\Omega)}+\left(\int_{Q}|u(x)-u(y)|^{p}K(x-y)dxdy\right)^{\frac{1}{p}}.\nonumber
\end{eqnarray}
We define the subspace $X_0$ of $X$ as 
$$X_0=\{u\in X:u=0~\text{a.e. in}~\mathbb{R}^N\setminus\Omega\}$$ equipped with the norm 
$$\|u\|=\left(\int_{Q}|u(x)-u(y)|^{p}K(x-y)dxdy\right)^{\frac{1}{p}}.$$
The space $X_0$ is a Banach and a reflexive space (refer Lemma 2.4 of \cite{xiang1} and Theorem 1.2 of \cite{adam}).
We will denote the usual fractional Sobolev space by $W_0^{s,p}(\Omega)$ equipped with norm 
$$\|u\|_{W^{s,p}(\Omega)}=\|u\|_{L^p(\Omega)}+\left(\int_{\Omega\times\Omega}|u(x)-u(y)|^{p}K(x-y)dxdy\right)^{\frac{1}{p}}.$$
Note that the norms $\|\cdot\|_{X}$ and $\|\cdot\|_{W^{s,p}}$ are not equivalent when $K(x)=|x|^{-N-ps}$ as $\Omega\times\Omega$ is strictly contained in $Q$. This makes the space $X_0$ different from the usual fractional Sobolev space. Thus the fractional Sobolev space is insufficient for dealing with our problem from the variational method point of view. For a general kernel $K$ satisfying \eqref{kernel} we have $X_0\subset\{u\in W^{s,p}(\mathbb{R}^N):u(x)=0~\text{a.e. in}~\mathbb{R}^N\setminus\Omega\}$. For $K(x)=\frac{1}{|x|^{N+sp}}$, we have $X_0=\{u\in W^{s,p}(\mathbb{R}^N):u(x)=0~\text{a.e. in}~\mathbb{R}^N\setminus\Omega\}=W_0^{s,p}(\mathbb{R}^N)$.\\
We now define the definition of a solution to \eqref{main} in a weaker sense. 
\begin{definition}\label{weaksoln}
A function $u\in X_0$ is a weak solution to \eqref{main} if $h|u|^{-\gamma}\phi\in L^1(\Omega)$ and
\begin{eqnarray}
(a+b\|u\|^p)\int_{Q}|u(x)-u(y)|^{p-2}(u(x)-u(y))(\phi(x)-\phi(y))K(x-y)dxdy\nonumber\\
-\lambda\int_{\Omega}g(x)|u|^{p-2}u\phi dx-\mu\int_{\Omega}h(x)|u|^{-\gamma-1}u\phi dx-\int_{\Omega}F(x,u)dx=0\nonumber
\end{eqnarray}
for all $\phi\in X_0$.
\begin{remark}\label{duality_pair}
We will sometimes denote $$\langle u,v \rangle=\int_{Q}|u(x)-u(y)|^{p-2}(u(x)-u(y))(v(x)-v(y))K(x-y)dxdy$$
\end{remark}
\begin{remark}
Henceforth, we will mean a weak solution whenever we use the term solution.	
\end{remark}
\begin{remark}
{Throughout the article any constant will be denoted by alphabets $C$ with a prefix/suffix.}
\end{remark}
\end{definition}
With these developements, we are now in a position to state our main result(s).
\begin{theorem}\label{mainthm1}
Let $K:\mathbb{R}^N\setminus\{0\}\rightarrow (0,\infty)$ be a function as above in \eqref{kernel}; $(i)-(iii)$, and let conditions $(f_1)-(f_5)$ hold. Then, for any $\lambda>0$ and for small enough $\mu_0>0$ such that for any $\mu\in(0,\mu_0)$, the problem in \eqref{main} has infinitely many solutions in $X_0$ with unbounded energy.	
\end{theorem}
\begin{theorem}\label{mainthm2}
	Let $K:\mathbb{R}^N\setminus\{0\}\rightarrow (0,\infty)$ be a function as above in \eqref{kernel}; $(i)-(iii)$ and let conditions $(f_1)-(f_4)$ hold. If condition $(f_6)$ is considered instead of $(f_5)$ then also the conclusion of the Theorem \ref{mainthm1} holds.	
\end{theorem}
\noindent The next theorem stated is a {\it $C^1$ versus $W_0^{s,p}$} analysis when the kernel $K(x)=|x|^{-N-sp}$ is considered. The functional will still be denoted by $I$.
\begin{theorem}\label{regularity holder}
	Let $\Omega\subset\mathbb{R}^N$ be a bounded domain with a $C^{1,1}$-boundary and let $u_0\in C^1(\bar{\Omega})$ which satisfies 
	\begin{equation}\label{bord}
	u_0\geq C\mbox{d}(x,\partial\Omega)^s\mbox{ for some }C>0
	\end{equation}
	be a local minimizer of $I$ (defined later in this section) in $C^1(\bar{\Omega})$ topology; that is, there exists $\epsilon>0$ such that, $u\in C^1(\bar{\Omega}),\;\|u-u_0\|_{C^1(\bar{\Omega})}<\epsilon$ implies $I(u_0)\leq I(u)$. Then, $u_0$ is a local minimum of $I$ in $W^{s,p}_0(\Omega)$ as well.
\end{theorem}
\noindent We quickly recall the following eigenvalue problem \cite{ren1}. 
\begin{eqnarray}\label{eigenprob}
\mathfrak{L}_p^su&=&\lambda|u|^{p-2}u~\text{in}~\Omega\nonumber\\
u&=&0~\text{in}~\mathbb{R}^N\setminus\Omega.
\end{eqnarray}
This has a divergent sequence of positive eigenvalues
$$0< \lambda_1\leq \lambda_2\leq\cdots\leq\lambda_n\leq\cdots$$
which has eigenvectors say $(e_n)_{n\in\mathbb{N}}$. Refer to Proposition $9$ of \cite{serva3} where it has been shown that this sequence $(e_n)_{n\in\mathbb{N}}$ can be so chosen that it provides an orthonormal basis in $L^p(\Omega)$ and an orthogonal basis in $X_0$.\\
We first define $I:X_0\rightarrow \mathbb{R}$ as
\begin{eqnarray}\label{energy_functional}
I(u)&=&A(u)-B(u)-C(u)
\end{eqnarray}
where 
\begin{align}
\begin{split}
A(u)&=\frac{a}{p}\|u\|^p+\frac{b}{2p}\|u\|^{2p}\\
B(u)&=\frac{\lambda}{p}\int_{\Omega}g(x)|u|^pdx\\
C(u)&=\frac{\mu}{1-\gamma}\int_{\Omega}h(x)|u|^{1-\gamma}dx+\int_{\Omega}F(x,u)dx\\
\text{and}~F(x,u)&=\int_{0}^uf(x,s)ds.
\end{split}
\end{align}
From here {onwards} we will denote $\|\cdot\|_{L^p(\Omega)}$ as $\|\cdot\|_p$. To our dissatisfaction, the functional $I$ is not a $C^1$ functional. However we redefine this functional $I$ as follows. We define
\[   
\bar{f}(x,t) = 
\begin{cases}
\mu h(x)|t|^{-\gamma-1}t+\tilde{f}(x,t), &~\text{if}~|t|>\underline{u}_{\mu}\\
\mu h(x)\underline{u}_{\mu}^{-\gamma}+\tilde{f}(x,\underline{u}_{\mu}),&~\text{if}~|t|\leq \underline{u}_{\mu}
\end{cases}\]
where {$\tilde{f}(x,t)=\lambda g(x)|t|^{p}+f(x,t)$} and $\underline{u}_{\mu}$ is a solution to
\begin{eqnarray}\label{auxprob}
\left(a+b\int_{\mathbb{R}^{2N}}|u(x)-u(y)|^{p}K(x-y)dxdy\right)\mathfrak{L}_p^su&=&\mu h(x)u^{-\gamma},~\text{in}~\Omega\nonumber\\
u&>&0,~\text{in}~\Omega\nonumber\\
u&=&0,~\text{in}~\mathbb{R}^N\setminus\Omega\nonumber\\
\end{eqnarray}
whose existence can be guaranteed from Lemma \ref{existence_positive_soln}.
Let $\bar{F}(x,s)=\int_{0}^{s}\bar{f}(x,t)ds$. We now define the functional as follows.
 \begin{eqnarray}\label{modi_func}
 \bar{I}(u)&=&\frac{a}{p}\|u\|^p+\frac{b}{2p}\|u\|^{2p}-\int_{\Omega}\bar{F}(x,u)dx.
 \end{eqnarray}
 The way the functional has been defined, it is easy to see that the critical points of the functional in \eqref{modi_func} are also the critical points of the functional \eqref{energy_functional}. Most importantly, the functional $\bar{I}$ which is defined in \eqref{modi_func} is $C^1$ which allows us to use the variational methods. {Further}
 \begin{eqnarray}\label{der_modi_func}
\langle \bar{I}'(u),\phi \rangle&=&(a+b\|u\|^p)\int_{Q}|u(x)-u(y)|^{p-2}(u(x)-u(y))(\phi(x)-\phi(y))K(x-y)dxdy\nonumber\\
 & &-\int_{\Omega}\bar{f}(x,u)\phi dx
 \end{eqnarray}
 for all $\phi\in X_0$.\\
 \begin{remark}\label{key_obs}It can easily seen that if $u$ is a weak solution to \eqref{auxprob} with $u>\underline{u}_{\lambda}$, then $u$ is also a weak solution to \eqref{main}. \end{remark}
 The following lemma will be used in  this work.
 \begin{lemma}\label{embres}
 (Refer Lemma 2.4 \cite{xiang1}) Let the kernel $K$ be as above in \eqref{kernel}. We then have the following.
 \begin{enumerate}
 \item For any $r\in[1,p_s^*)$, the embedding $X_0\hookrightarrow L^r(\Omega)$ is compact when $\Omega$ is bounded with smooth enough boundary.
 \item For all $r\in[1,p_s^*]$, the embedding $X_0\hookrightarrow L^r(\Omega)$ is continuous.
 \end{enumerate}
 \end{lemma}
\noindent Following is the definition of Cerami condition (Definition 2.1 \cite{soni1}).
\begin{definition}\label{cerami}
[Cerami condition]~Let $B$ be an infinite dimensional Banach space with respet to $\|\cdot\|_B$ and $I$ be a $C^1(B,\mathbb{R})$ functional.  $I$ is said to satisfy the $(Ce)_c$ at level $c\in\mathbb{R}$, if any sequence $(u_n)\subset B$ for which $I(u_n)\rightarrow c$ in $B$, $(1+\|u_n\|_{B})I'(u_n)\rightarrow 0$ in $B'$, the dual space of $B$, as $n\rightarrow\infty$, then there exists a strongly convergent subsequence of $(u_{n_k})$ of $(u_n)$ in $B$. 
\end{definition}
\begin{remark}
Henceforth, a subsequence of any sequence, say $(v_n)$, will also be denoted by $(v_n)$.
\end{remark}
\noindent We now give the symmetric mountain pass theorem \cite{col}.
\begin{theorem}\label{symmMPT}
(Symmetric mountain pass theorem) Let $B$ be as in Definition \ref{cerami} and $Y$ be a finite dimensional Banach space such $B=Y \bigoplus Z$. For any $c>0$ if $I\in C^1(B,\mathbb{R})$ satisfies $(Ce)_c$ and 
\begin{enumerate}
\item $I$ is even and $I(0)=0$ for all $u\in B$ 
\item There exists $r>0$ such that $I(u)\geq R$ for all $u\in B_r(0)=\{u\in B:\|u\|_B\leq r\}$
\item For any finite dimensional subspace $\bar{B}\subset B$, there exists $r_0=r(\bar{B})>0$ such that $I(u)\leq 0$ on $\bar{B}\setminus B_{r_0}(0_{\bar{B}})$, where $0_{\bar{B}}$ is the null vector in $\bar{B}$
\end{enumerate}
then there exists an unbounded sequence of critical values of $I$ characterized by a minimax argument.
\end{theorem}
\section{Auxiliary and main result}
We begin this section by proving a few auxiliary lemmas.
\begin{lemma}\label{Ce_lemma}
Let $(f_1)$ hold. Then any bounded sequence $(u_n)$ in $X_0$ which satisfies $(1+\|u_n\|)\bar{I}'(u_n)\rightarrow 0$ as $n\rightarrow\infty$ possesses a strongly convergent subsequence in $X_0$.
\end{lemma}
\begin{proof}
Let $(u_n)$ be a bounded sequence in $X_0$ with $|u_n|>\underline{u}_{\mu}$. Since $X_0$ is reflexive we have 
\begin{align}\label{conv1}
\begin{split}
& u_n\rightharpoonup u~\text{in}~X_0,\\
& u_n\rightarrow u~\text{in}~L^r(\Omega), 1\leq r< p_s^*,\\
& u_n\rightarrow u~\text{a.e. in}~\Omega.
\end{split}
\end{align}
All we need to prove is that $u_n\rightarrow u$ strongly in $X_0$. By the H\"{o}lder's inequality and $(f_1)$ we have
\begin{align}\label{holdineq1}
\begin{split}
0\leq \int_{\Omega}|f(x,u_n)|(u_n-u)dx&\leq\int_{\Omega}C(1+|u_n|^{q-1})(u_n-u)dx\\
&\leq C (|\Omega|^{\frac{q-1}{q}}+\|u_n\|_q^{q-1})\|u_n-u\|_{q}.
\end{split}
\end{align}
By \eqref{conv1} we obtain
$$\underset{n\rightarrow\infty}{\lim}\int_{\Omega}|f(x,u_n)|(u_n-u)dx=0.$$
Further we have from the H\"{o}lder's inequality that 
$$\int_{\Omega}|u_n|^{p-2}u_n(u_n-u)dx\leq\|u_n\|_p^{p-1}\|u_n-u\|_p.$$
Therefore again from \eqref{conv1} we have
$$\underset{n\rightarrow\infty}{\lim}\int_{\Omega}|u_n|^{p-2}u_n(u_n-u)dx=0.$$
We observe that $u$ cannot be zero over a subset of $\Omega$ of non-zero measure. For if it is so, then $\langle \bar{I}'(u_n),\phi \rangle\rightarrow-\infty$ for a suitable $\phi\in X_0$. But then we have from the given condition that $(1+\|u_n\|)\bar{I}'(u_n)=o(1)$ as $n\rightarrow\infty$ and this gives rise to an absurdity of $0=-\infty$.\\ 
Note that since $|u_n|>\underline{u}_{\mu}$, we have
 \begin{eqnarray}
 \underset{n\rightarrow\infty}{\lim}||u_n|^{-\gamma-1}u_n(u_n-u)|&=&
\underset{n\rightarrow\infty}{\lim}||u_n|^{1-\gamma}-u|u_n|^{-\gamma-1}u_n|=0
\end{eqnarray}
and hence $$\underset{n\rightarrow\infty}{\lim}\int_{\Omega}|u_n|^{-\gamma-1}u_n(u_n-u)dx=0.$$ 
We now define a linear functional 
$$J_v(u)=\int_{\mathbb{R}^N}|v(x)-v(y)|^{p-2}(v(x)-v(y))(u(x)-u(y))K(x-y)dxdy.$$
The H\"{o}lder inequality yields 
$$|J_{v}(u)|\leq\|v\|^{p-1}\|u\|$$ for every $u\in X_0$
thereby implying that $J$ is a continuous linear functional on $X_0$. Therefore
$$\underset{n\rightarrow\infty}{\lim}J_v(u_n-u)=0.$$
From the discussion so far in the proof of this theorem and since $u_n\rightharpoonup u$ in $X_0$, we conclude that $\langle\bar{I}'(u_n),u_n-u\rangle\rightarrow 0$ as $n\rightarrow\infty$. Also $(1+\|u_n\|_{X_0})\bar{I}'(u_n)\rightarrow 0$ in $X_0'$, the dual space of $X_0$. Therefore, we have
\begin{align}
\begin{split}
o(1)=&\langle\bar{I}'(u_n),u_n-u\rangle-\lambda\int_{\Omega}g(x)|u_n|^{p-2}u_n(u_n-u)dx\\
&-\mu\int_{\Omega}h(x)|u_n|^{-\gamma-1}u_n(u_n-u)dx-\int_{\Omega}f(x,u_n)(u_n-u)dx\\
=&(a+b\|u_n\|^p)J_{u_n}(u_n-u)+o(1)~\text{as}~n\rightarrow\infty.
\end{split}
\end{align}
Hence, by the boundedness of $(u_n)$ in $X_0$ and $\underset{n\rightarrow\infty}{\lim}J_v(u_n-u)=0$ we have 
\begin{eqnarray}
\underset{n\rightarrow\infty}{\lim}(J_{u_n}(u_n-u)-J_{u}(u_n-u))&=&0.
\end{eqnarray}
Recall the Simon inequalities which is as follows.
\begin{align}\label{simon}
\begin{split}
|U-V|^{p}&\leq C_p(|U|^{p-2}U-|V|^{p-2}V).(U-V),~p\geq 2\\
|U-V|^{p}&\leq C_p'(|U|^{p-2}U-|V|^{p-2}V)^{\frac{p}{2}}.(|U|^p+|V|^p)^{\frac{2-p}{2}},~1<p<2
\end{split}
\end{align}
for all $U,V\in\mathbb{R}^N$, $C_p, C_p'>0$ are constants.\\
When $p\geq 2$, we have 
\begin{align}
\begin{split}
\|u_n-u\|^p\leq& C_p\int_{\mathbb{R}^{2N}}[|u_n(x)-u_n(y)|^{p-2}(u_n(x)-u_n(y))-|u(x)-u(y)|^{p-2}(u(x)-u(y))]\\
&\times [(u_n(x)-u_n(y))-(u(x)-u(y))]K(x-y)dxdy\\
=&C_p[J_{u_n}(u_n-u)-J_u(u_n-u)]\rightarrow 0~\text{as}~n\rightarrow\infty.
\end{split}
\end{align}
When $1<p<2$
\begin{align}
\begin{split}
\|u_n-u\|^p\leq& C_p'[J_{u_n}(u_n-u)-J_{u}(u_n-u)]^{\frac{p}{2}}(\|u_n\|^p+\|u\|^p)^{\frac{2-p}{2}}\\
\leq&C_p'[J_{u_n}(u_n-u)-J_{u}(u_n-u)]^{\frac{p}{2}}(\|u_n\|^{\frac{2-p}{2}}+\|u\|^{\frac{2-p}{2}})\rightarrow 0~\text{as}~n\rightarrow\infty. 
\end{split}
\end{align}
Thus $u_n\rightarrow u$ strongly in $X_0$ as $n\rightarrow\infty$.
\end{proof}
\begin{lemma}
Let $(f_1), (f_3), (f_5)$ hold. Then the functional $\bar{I}$ satisfies the $(Ce)_c$ condition.
\end{lemma}
\begin{proof}
Let $(f_5)$ hold. From the monotonicity of $t \mapsto \frac{f(x,t)}{t^{2p-1}}$ there exists $C_1>0$ such that
\begin{eqnarray}
\mathfrak{G}(x,s)&\leq&\mathfrak{G}(x,t)+C,
\end{eqnarray}
for all $x\in\Omega$ and $0\leq |s|\leq |t|$. Refer $(f_6)$ for the definition of $\mathfrak{G}(.,.)$. Let $(u_n)$ be a Cerami sequence in $X_0$. Thus for $c\in\mathbb{R}$
\begin{align}\label{Ce_cond1}
\begin{split}
\bar{I}(u_n)\rightarrow c~\text{in}~X_0\\
(1+\|u_n\|)\bar{I}'(u_n)\rightarrow 0~\text{in}~X_0'
\end{split}
\end{align}
as $n\rightarrow\infty$. All we need to show is that the sequence $(u_n)$ is bounded in $X_0$ and the conclusion will follow from the Lemma \ref{Ce_lemma}.\\
Suppose not, i.e. {there exists a subsequence} such that $\|u_n\|_{X_0}\rightarrow\infty$. By the second condition of \eqref{Ce_cond1} we have 
$$\bar{I}'(u_n)\rightarrow 0$$ as $n\rightarrow\infty$.
\noindent Hence 
$$\|u_n\|\left\langle \bar{I}'(u_n),\frac{u_n}{\|u_n\|} \right\rangle\rightarrow 0$$ as $n\rightarrow\infty$. We define $\xi_n=\frac{u_n}{\|u_n\|}$ so that $\|\xi_n\|=1$. Thus $(\xi_n)$ is a bounded sequence and hence 
\begin{align}\label{conv_rule}
\begin{split}
\xi_n\rightarrow \xi~\text{in}~L^{p}(\Omega)\\
\xi_n\rightarrow \xi~\text{in}~L^{q}(\Omega)\\
\xi_n\rightarrow \xi~\text{a.e. in}~ \Omega.
\end{split}
\end{align}
Further, from Lemma A.1  of \cite{xiang5} there exists a function $\alpha(x)$ such that 
$$|\xi_n(x)|\leq\alpha(x)~\text{in}~\mathbb{R}^N.$$
This leads to the consideration of two cases, {\it viz.} $\xi=0$ and $\xi\neq 0$. Since $\bar{I}$ is a $C^1$ functional, therefore
$\bar{I}(\alpha_n u_n)=\underset{\alpha \in[0,1]}{\max}\bar{I}(\alpha u_n)$ makes sense.\\
Let $\xi=0$ and define $h_T=\left(\frac{4p T}{b}\right)^{\frac{1}{2p}}$ such that $\frac{h_{T}}{\|u_n\|}\in(0,1)$ for $T\in\mathbb{N}$ and sufficiently large $n$, say for $n\geq n(T)$. Since $\xi=0$
and by \eqref{conv_rule} we have 
\begin{eqnarray}\label{conv2}
\int_{\Omega}|h_T\xi_n|^pdx\rightarrow 0
\end{eqnarray}
as $n\rightarrow\infty$. By the continuity of $F$ we obtain 
\begin{eqnarray}\label{conv3}
F(x,h_T\xi_n(x))\rightarrow F(x,h_T\xi(x))=F(x,0)~\text{in}~\Omega~\text{as}~n\rightarrow\infty.
\end{eqnarray}
From $(f_1)$ and $|\xi_n(x)|\leq\alpha(x)~\text{in}~\mathbb{R}^N$ in combination with the H\"{o}lder inequality we get 
\begin{eqnarray}\label{con4}
|F(x,h_T\xi_n)|&\leq&C|h_T\alpha(x)|+\frac{C}{q}|h_T\alpha(x)|^q~\in L^1(\Omega)
\end{eqnarray}
for any $n,T\in\mathbb{N}$. Therefore from the Lebesgue dominated convergence theorem we get
\begin{eqnarray}\label{conv5}
F(.,h_T\xi_n(.))\rightarrow F(.,h_T\xi(.))~\text{in}~L^1(\Omega)~\text{as}~n\rightarrow\infty.
\end{eqnarray}
for any $T\in\mathbb{N}$. Therefore, since $F(x,0)=0$ for each $x\in\Omega$, $\xi=0$ and by \eqref{conv5}, we have $$\int_{\Omega}F(x,h_T\xi_n(x))dx\rightarrow 0~\text{as}~n\rightarrow\infty.$$
We further have 
\begin{eqnarray}\label{conv6}
\bar{I}(\alpha_nu_n)&\geq&\bar{I}\left(\frac{h_T}{\|u_n\|}u_n\right)\nonumber\\
&=&\bar{I}(h_T\xi_n)\nonumber\\
&\geq&\frac{a}{p}\|h_T\xi_n\|^p+\frac{b}{2p}\|h_T\xi_n\|^{2p}-\frac{\lambda\|g\|_{\infty}}{p}\int_{\Omega}|h_T\xi_n|^pdx\nonumber\\
& &-\frac{\mu\|h\|_{\infty}}{1-\gamma}\int_{\Omega}|h_T\xi_n|^{1-\gamma}dx-\int_{\Omega}F(x,h_T\xi_n)dx\nonumber\\
&\geq&\frac{b}{2p}\|h_T\xi_n\|^{2p}+o(1)=2T
\end{eqnarray}
as $n\rightarrow\infty$ and for any $T\in\mathbb{N}$.
Therefore \begin{eqnarray}\label{conv6'}\bar{I}(\alpha_nu_n)\rightarrow\infty\end{eqnarray} as $n\rightarrow\infty$. We now show that 
$$\underset{n\rightarrow\infty}{\lim}\sup \bar{I}(\alpha_nu_n)\leq\beta_0$$
for some $\beta_0>0$. Since $\int_{\Omega}\frac{|u_n|^p}{\|u_n\|^p}dx\leq C$ and $\frac{1}{\|u_n\|^p}\leq C'$ for all $n\in\mathbb{N}$ we have that $\int_{\Omega} |u_n|^pdx\leq C''$ for all $n\in\mathbb{N}$. Furthermore, since the boundary $\partial\Omega$ is Lipschitz continuous we have that 
\begin{align}\label{conv7}
\begin{split}
&\left|\frac{\lambda}{2p}\int_{\Omega}g(x)|\alpha_nu_n|^pdx\right|\leq C(\lambda,p)\|h\|_{\infty}<\infty\\
&\left|\mu\left(\frac{1}{1-\gamma}-\frac{1}{2p}\right)\int_{\Omega}h(x)(\alpha_nu_n)^{1-\gamma}dx\right|\leq C(\mu,\gamma,p)\|h\|_{\infty}<\infty.
\end{split}
\end{align} 
Since $\frac{d}{d\alpha}|_{\alpha=\alpha_n}\bar{I}(\alpha u_n)=0$ for all $n$ we have
\begin{eqnarray}\label{conv8}
\langle \bar{I}'(\alpha_nu_n),\alpha_nu_n \rangle&=&\alpha_n\frac{d}{d\alpha}|_{\alpha=\alpha_n}\bar{I}(\alpha u_n)=0.
\end{eqnarray}
We further have
\begin{eqnarray}
\bar{I}(\alpha_nu_n)&=&\bar{I}(\alpha_nu_n)-\frac{1}{2p}\langle \bar{I}'(\alpha_nu_n),\alpha_nu_n \rangle\nonumber\\
&=&\frac{a}{2p}\|\alpha_n u_n\|^p-\frac{\lambda}{2p}\int_{\Omega}g(x)|\alpha_nu_n|^pdx-\mu\left(\frac{1}{1-\gamma}-\frac{1}{2p}\right)\int_{\Omega}h(x)(\alpha_nu_n)^{1-\gamma}dx\nonumber\\
& &-\int_{\Omega}F(x,\alpha_nu_n)dx+\frac{1}{2p}\int_{\Omega}f(x,\alpha_nu_n)\alpha_nu_ndx\nonumber\\
&\leq&\frac{a}{2p}\|\alpha_nu_n\|^p+C(\lambda,p)+C(\mu,\gamma,p)+\int_{\Omega}\mathfrak{G}(x,\alpha_nu_n)dx\nonumber\\
&\leq&\frac{a}{2p}\|\alpha_nu_n\|^p+C(\lambda,p)+C(\mu,\gamma,p)+\int_{\Omega}\mathfrak{G}(x, u_n)dx+C|\Omega|\nonumber
\end{eqnarray}
\begin{eqnarray}\label{conv9}
&=&\frac{a}{2p}\|\alpha_nu_n\|^p+C(\lambda,p)+C(\mu,\gamma,p)+\int_{\Omega}\mathfrak{G}(x, u_n)dx+C|\Omega|\nonumber\\
& &-\frac{\lambda}{2p}\int_{\Omega}g(x)|\alpha_nu_n|^pdx+\frac{\lambda}{2p}\int_{\Omega}g(x)|\alpha_nu_n|^pdx\nonumber\\
& &-\mu\left(\frac{1}{1-\gamma}-\frac{1}{2p}\right)\int_{\Omega}h(x)(\alpha_nu_n)^{1-\gamma}dx
+\mu\left(\frac{1}{1-\gamma}-\frac{1}{2p}\right)\int_{\Omega}h(x)(\alpha_nu_n)^{1-\gamma}dx\nonumber\nonumber\\
& &+C|\Omega|\nonumber\\
&\leq&\overline{I}(\alpha_nu_n)-\frac{1}{2p}\langle \overline{I}'(\alpha_nu_n),\alpha_nu_n\rangle+2C(\lambda,p)+2C(\mu,\gamma,p)+C|\Omega|\nonumber\\
&=&c+o(1)+2C(\lambda,p)+2C(\mu,\gamma,p)+C|\Omega|<\infty~\text{as}~n\rightarrow\infty.
\end{eqnarray}
where $|\Omega|$ is the Lebesgue measure of $\Omega$. This is a contradiction to \eqref{conv6'}. Thus $(u_n)$ is bounded in $X_0$.\\
Let $\xi\neq 0$. Define 
$$A=\{x\in\Omega:\xi(x)\neq 0\}.$$
Therefore we have
\begin{eqnarray}\label{conv10}
|u_n(x)|&=&|\xi_n(x)|\|u_n\|\rightarrow\infty~\text{in}~A~\text{as}~n\rightarrow\infty.
\end{eqnarray}
Further from $(f_3)$
\begin{eqnarray}\label{conv11}
\frac{F(x,u_n(x))}{\|u_n\|^{2p}}&=&\frac{F(x,u_n(x))}{|u_n(x)|^{2p}}\frac{|u_n(x)|^{2p}}{\|u_n\|^{2p}}\nonumber\\
&=&\frac{F(x,u_n(x))}{|u_n(x)|^{2p}}|\xi_n|^{2p}\rightarrow\infty~\text{in}~A~\text{as}~n\rightarrow\infty.
\end{eqnarray}
By the Fatou's lemma
\begin{eqnarray}\label{conv12}
\int_{A}\frac{F(x,u_n(x))}{\|u_n\|^{2p}}dx\rightarrow\infty.
\end{eqnarray}
Let us now analyse the case over $\Omega\setminus A$. From $(f_3)$ again we have $$\underset{|t|\rightarrow\infty}{\lim}F(x,t)=\infty$$
for $x\in\Omega$. Therefore for arbitrary $M>0$ there exists $t'$ such that 
\begin{eqnarray}\label{conv13}
F(x,t)\geq M~\text{whenever}~|t|\geq t', x\in\Omega. 
\end{eqnarray}
Hence
\begin{eqnarray}\label{conv14}
F(x,t)&\geq\min\left\{M, \underset{(x,t)\in\Omega\times[-t',t']}{\min}\{F(x,t)\}\right\}=M'.
\end{eqnarray}
So we get 
\begin{eqnarray}\label{conv15}
\underset{|t|\rightarrow\infty}{\lim}\int_{\Omega\setminus A}\frac{F(x,u_n(x))}{\|u_n\|^{2p}}dx\geq 0.
\end{eqnarray}
Also
\begin{eqnarray}\label{conv15'}
0\leq\frac{\mu}{1-\gamma}\int_{\Omega}h(x)\frac{|u_n|^{1-\gamma}}{\|u_n\|^{2p}}dx&\leq&\frac{\mu\|h\|_{\infty}}{1-\gamma}\|u_n\|^{1-\gamma-2p}=o(1).
\end{eqnarray}
By the variational characterization (refer \cite{ren1}) of $\lambda_j$, the $j$-th eigenvalue of $\mathfrak{L}_p^s$ is
$$\lambda_j=\underset{u\in X_0\setminus\{0\}}{\min}\left\{\frac{\int_{\mathbb{R}^{2N}}|u(x)-u(y)|^{p}K(x-y)dxdy}{\int_{\Omega}|u(x)|^pdx}\right\}.$$ 
Now since
\begin{align}\label{conv16}
\begin{split}
o(1)&=\frac{\overline{I}(u_n)}{\|u_n\|^{2p}}\\
=&\frac{a}{p\|u_n\|^{p}}+\frac{b}{2p}-\frac{\lambda}{p}\int_{\Omega}g(x)\frac{|u_n(x)|^p}{\|u_n\|^{2p}}dx\\
&-\frac{\mu}{1-\gamma}\int_{\Omega}h(x)\frac{|u_n|^{1-\gamma}}{\|u_n\|^{2p}}dx-\int_{A}\frac{F(x,u_n)}{\|u_n\|^{2p}}dx-\int_{\Omega\setminus A}\frac{F(x,u_n)}{\|u_n\|^{2p}}dx\\
\leq&o(1)+\frac{b}{2p}+\frac{\lambda\|g\|_{\infty}}{p\lambda_j}\frac{1}{\|u_n\|^{p}}-\frac{\mu}{1-\gamma}\int_{\Omega}h(x)\frac{|u_n|^{1-\gamma}}{\|u_n\|^{2p}}dx-\int_{A}\frac{F(x,u_n)}{\|u_n\|^{2p}}dx-\int_{\Omega\setminus A}\frac{F(x,u_n)}{\|u_n\|^{2p}}dx\\
\leq&-\infty
\end{split}
\end{align}
where the last step is due to \eqref{conv12}, \eqref{conv15}, \eqref{conv15'}. This is again an absurdity. Thus the sequence $(u_n)$ is bounded in $X_0$ and hence by the Lemma \ref{Ce_lemma} we conclude that $(u_n)$ possesses a strongly convergent subsequence in $X_0$.
\end{proof}
\noindent We now prove the results stated in the Theorems \ref{mainthm1} and \ref{mainthm2}. For this let us develope some prerequisites. It is well known that the space $X_0$ is a Banach space and we have that 
$$X_0=\underset{i\geq 1}{\bigoplus}X_i$$
where $X_{i}=\text{span}\{e_j\}_{j\geq i}$. Define
$$Y_m=\underset{1\leq j \leq m}{\bigoplus}X_j$$
$$Z_m=\underset{j\geq m}{\bigoplus}X_j.$$
Clearly, $Y_m$ is a finite dimensional subspace of $X_0$, for each $m$.
\begin{theorem}\label{conv17}
Let $\kappa\in[1,p_s^*)$. We have 
$$\zeta_m(\kappa)=\sup\{\|u\|_{\kappa}:u\in Z_m,\|u\|=1\}\rightarrow 0$$ as $m\rightarrow\infty$.
\end{theorem}
\begin{proof}
From the definition of $(Z_m)$ we have that $Z_{m+1}\subset Z_m$ and thus $0\leq\zeta_{m+1}\leq\zeta_m$. This implies that $\zeta_m\rightarrow\zeta\geq 0$ as $m\rightarrow\infty$. Further by the definition of {\it supremum} for every $m$ there exists $u_m\in Z_m$ such that $\|u_m\|=1$ and $\|u_m\|_{\kappa}>\frac{\zeta_m}{2}$. By the reflexivity of $X_0$ and $\|u_m\|=1$ we have that $u_m\rightharpoonup 0$ in $X_0$. By the embedding results from Lemma \ref{embres} we have that $u_m\rightarrow 0$ as $m\rightarrow\infty$ in $L^{\kappa}(\Omega)$ for any $\kappa\in[1,p_s^*)$. Thus we conclude that $\zeta=0$.
\end{proof}
\noindent {\it Proof of Theorem \ref{mainthm1}}:~Since $Y_m$ is a finite dimensional space, hence the norms $\|\cdot\|$ and $\|\cdot\|_{\kappa}$ are equivalent for $\kappa\in[1,p_s^*)$. Mathematically this means that there exists $C_1, C_2>0$ such that 
\begin{align}\label{norm_equiv_fin_dim_sp}
C_1\|u\|_{X_0}\leq\|u\|_{\kappa}\leq C_2\|u\|_{X_0},~\text{for any}~\kappa\in[1,p_s^*).
\end{align}
Observe that $\bar{I}(u)=0$. If $\frac{\lambda}{a}>\lambda_1$, then from the definition of $\lambda_j$ given in the previous theorem, for any $\lambda,\mu\in\mathbb{R}^+$ $\exists j>1$ such that $\lambda_j$ such that $\frac{\lambda}{a}\in[\lambda_{j-1},\lambda_j)$, $\frac{\mu}{a}\in[\lambda_{k-1},\lambda_k)$. Further, from $(f_1)$ and $(f_4)$, for every $\epsilon>0$ there exists $C_{\epsilon}>0$ such that 
\begin{eqnarray}\label{conv18}
F(x,t)&\leq&\frac{\epsilon}{p}|t|^{p}+\frac{C_{\epsilon}}{q}|t|^q,
\end{eqnarray}
for any $(x,t)\in\Omega\times\mathbb{R}$. By the definition of $\zeta_m$ in Lemma \ref{conv17}, fix $\epsilon>0$ and choose $m'\geq 1$ such that 
\begin{align}\label{conv19}
\begin{split}
\|u\|_p^p&\leq\frac{a-\frac{\lambda}{\lambda_j}}{2\epsilon}\|u\|^p\\
\|u\|_q^q&\leq\frac{q(a-\frac{\lambda}{\lambda_j})}{2p C_{\epsilon}}\|u\|^q~\text{for every}~u\in Z_{m'}.
\end{split}
\end{align}
Choose $\|u\|=r<1$ sufficiently small, in Theorem \ref{symmMPT}. Since $q>p$, we have
\begin{align*}
\begin{split}
\bar{I}(u)&=\frac{a}{p}\|u\|^p+\frac{b}{2p}\|u\|^{2p}-\frac{\lambda}{p}\int_{\Omega}g(x)|u(x)|^pdx-\frac{\mu}{1-\gamma}\int_{\Omega}h(x)u^{1-\gamma}dx-\int_{\Omega}F(x,u)dx\\
&\geq \frac{a}{p}\|u\|^p-\frac{\lambda\|g\|_{\infty}}{p\lambda_j}\|u\|^p-\frac{C\mu\|h\|_{\infty}}{1-\gamma}\|u\|^{1-\gamma}-\frac{\epsilon}{p}\|u\|_p^{p}-\frac{C_{\epsilon}}{q}\|u\|_q^q
\end{split}
\end{align*}
\begin{align}\label{conv20}
\begin{split}
&\geq \left(\frac{a}{p}-\frac{\lambda\|g\|_{\infty}}{p\lambda_j}\right)\|u\|^p-\frac{C\mu\|h\|_{\infty}}{1-\gamma}\|u\|^{1-\gamma}-\frac{\left(a-\frac{\lambda}{\lambda_j}\right)}{2p}\|u\|^{p}-\frac{\left(a-\frac{\lambda}{\lambda_j}\right)}{2p}\|u\|^q\\
&\geq \frac{a-\frac{\lambda}{\lambda_j}}{2p}\|u\|^p-\frac{a-\frac{\lambda}{\lambda_j}}{2p}\|u\|^q-\frac{C\mu\|g\|_{\infty}}{1-\gamma}\|u\|^{1-\gamma}+\frac{\lambda}{p\lambda_j}(1-\|g\|_{\infty})\|u\|^p\\
&=\frac{a-\frac{\lambda}{\lambda_j}}{2p}(r^p-r^q)-\frac{C\mu\|h\|_{\infty}}{1-\gamma}r^{1-\gamma}+\frac{\lambda}{p\lambda_j}(1-\|g\|_{\infty})r^p=R>0
\end{split}
\end{align}
for sufficiently small $\mu>0$ (the choice of $r$ tackles the case of $\|g\|_{\infty}>1$). Finally, from $(f_3)$, there exists $C_0>\frac{b}{2pC^{2p}}$ (a possible choice of $C$, as we shall see later, is a Sobolev constant), $C_1>0$ such that 
\begin{eqnarray}\label{conv21}
F(x,t)&\geq&C_0|t|^{2p}
\end{eqnarray}
for any $x\in\Omega$ and $|t|>C_1$. From $(f_1)$ we have
\begin{eqnarray}\label{conv22}
|F(x,t)|&\leq&C(1+C_1^{q-1})|t|,~\text{for every}~x\in\Omega~\text{and}~|t|\leq C_1.
\end{eqnarray}
Let $C'=C(1+C_2^{q-1})>0$. Then we get
\begin{eqnarray}\label{conv23}
F(x,t)&\geq&C_0|t|^{2p}-C'|t|,~\text{for}~(x,t)\in\Omega\times\mathbb{R}.
\end{eqnarray}
By the equivalence of norm in $Y_m$ and \eqref{conv23}, we have
\begin{align}\label{conv24}
\begin{split}
\bar{I}(u)=&\frac{a}{p}\|u\|^p+\frac{b}{2p}\|u\|^{2p}-\frac{\lambda}{p}\int_{\Omega}g(x)|u|^pdx -\frac{\mu}{1-\gamma}\int_{\Omega}h(x)|u|^{1-\gamma}dx\\
&-\int_{\Omega}F(x,u)dx\\
\leq&\frac{a}{p}\|u\|^p+\frac{b}{2p}\|u\|^{2p}-\frac{\lambda}{p}\int_{\Omega}g(x)|u|^pdx -\int_{\Omega}F(x,u)dx\\
\leq&\frac{a}{p}\|u\|^p+\frac{b}{2p}\|u\|^{2p}-\frac{\lambda}{p\lambda_j}\|u\|^p-C_0\|u\|_{2p}^{2p}+C'\|u\|_1\\
\leq&\frac{a}{p}\|u\|^p+\left(\frac{b}{2p}-C_0 C^{2p}\right)\|u\|^{2p}+C_2C'\|u\|.
 \end{split}
\end{align}
Thus for a sufficiently large $r_0=r(\bar{X})$, we have $\bar{I}(u)\leq 0$ whenever $\|u\|\geq r_0$. Hence, by the Theorem \ref{symmMPT}, there exists an unbounded sequence of critical values of $\bar{I}$ characterized by a minimax argument. In other words, from the Remark \ref{key_obs} the problem in \eqref{auxprob} has infinitely many solutions and hence the problem \eqref{main} also has infinitely many solutions.\\\\
{\it Proof of Theorem \ref{mainthm2}}:~Suppose now $(f_6)$ holds. The proof follows {\it verbatim} of the Theorem \ref{mainthm1}, except that we need to prove the inequality in \eqref{conv9}. Thus we have 
\begin{align*}
\begin{split}
\frac{1}{\sigma}\bar{I}(\alpha_nu_n)=&\frac{1}{\sigma}\left(\bar{I}(\alpha_nu_n)-\frac{1}{2p}\langle \bar{I}'(\alpha_nu_n),\alpha_nu_n \rangle\right)\\
=& \frac{1}{\sigma}\left[\left(\frac{a}{2p} \right)\|\alpha_nu_n\|^p+\frac{\lambda}{2p}\int_{\Omega}g(x)|\alpha_nu_n|^pdx-\mu\left(\frac{1}{1-\gamma}-\frac{1}{2p}\right)\int_{\Omega}h(x)(\alpha_nu_n)^{1-\gamma}dx\right.  \\
&\left.-\int_{\Omega}F(x,\alpha_nu_n)dx+\frac{1}{2p}\int_{\Omega}f(x,\alpha_nu_n)\alpha_nu_ndx\right]\\
\leq&\frac{1}{\sigma}\left[\left(\frac{a}{2p} \right)\|\alpha_nu_n\|^p+\int_{\Omega}\mathfrak{G}(x,\alpha_nu_n)dx\right]+C(\lambda,p)\|g\|_{\infty}\\
\leq&\left(\frac{a}{2p} \right)\|\alpha_nu_n\|^p+\int_{\Omega}\mathfrak{G}(x,\alpha_nu_n)dx+\frac{1}{\sigma}\int_{\Omega}T(x)dx+C(\lambda,p)\|g\|_{\infty}\\
\leq&\left(\frac{a}{2p} \right)\|\alpha_nu_n\|^p+\int_{\Omega}\mathfrak{G}(x,\alpha_nu_n)dx+\frac{1}{\sigma}\int_{\Omega}T(x)dx+C(\lambda,p)\|g\|_{\infty}\\
&-\frac{\lambda}{2p}\int_{\Omega}|u_n|^pdx+\frac{\lambda}{2p}\int_{\Omega}|u_n|^pdx
\end{split}
\end{align*}
\begin{align}\label{conv25}
\begin{split}
\leq&\bar{I}_{u_n}-\frac{1}{2p}\langle \bar{I}'(\alpha_nu_n),\alpha_nu_n \rangle +2C(\lambda,p)\|g\|_{\infty}+\frac{1}{\sigma}\int_{\Omega}T(x)dx\\
\leq& c+o(1)+2C(\lambda,p)\|g\|_{\infty}+\frac{1}{\sigma}\int_{\Omega}T(x)dx<\infty.
\end{split}
\end{align}
{This completes the proof.}\\\\
\noindent We will now show that the solution $u$ to \eqref{main} is in $L^{\infty}(\Omega)$, i.e. bounded in $\Omega$ for $K(x)=|x|^{-N-sp}$. 
 Firstly, we will prove the following elementary inequality needed for the proof of the $L^\infty$ estimate.
 
 \begin{lemma}\label{ineq1}
 	(Lemma 5.1 in \cite{ghosh_jmp}) For all $a$, $b\in\mathbb{R}$, $\rho\geq p$, $p\geq 1$, $k>0$ we have
 	\begin{align*}
 	\begin{split}
 	\frac{p^p(\rho+1-p)}{\rho^p}&(a|a|_k^{\frac{\rho}{p}-1}-b|b|_k^{\frac{\rho}{p}-1})^p\\&\leq (a|a|_k^{\rho-1}-b|b|_k^{\rho-1})(a-b)^{p-1}
 	\end{split}
 	\end{align*}
 	with the assumption that $a \geq b$.
 \end{lemma}
 \begin{proof}
 	Define 
 	\[m(t)=\begin{cases}
 	\text{sgn}(t)|t|^{\frac{\rho}{p}-1}, & |t|<k \\
 	\frac{p}{\rho}\text{sgn}(t)k^{\frac{\rho}{p}-1}, & |t|\geq k.              
 	\end{cases}\]
 	Observe that 
 	\begin{align*}
 	\begin{split}
 	\int_{b}^{a}m(t)dt&=\frac{p}{\rho}(a|a|_k^{\frac{\rho}{p}-1}-b|b|_k^{\frac{\rho}{p}-1}).
 	\end{split}
 	\end{align*}
 	Similarly,
 	\begin{align*}
 	\begin{split}
 	\int_{b}^{a}m(t)^pdt&\leq\frac{1}{\rho+1-p}(a|a|_k^{\rho-p}-b|b|_k^{\rho-p}).
 	\end{split}
 	\end{align*}
 	On using the Cauchy-Schwartz inequality we obtain
 	\begin{align*}
 	\begin{split}
 	\left(\int_{b}^{a}m(t)dt\right)^p&\leq(a-b)^{p-1}\int_{b}^{a}h(t)^pdt.
 	\end{split}
 	\end{align*}
 	Thus 
 	\begin{align*}
 	\begin{split}
 	&\frac{p^p}{\rho^p}(a|a|_k^{\frac{\rho}{p}-1}-b|b|_k^{\frac{\rho}{p}-1})^p\\&=\left(\int_{b}^{a}m(t)dt\right)^p\\
 	&\leq(a-b)^{p-1}\int_{b}^{a}m(t)^pdt\\
 	&\leq\frac{(a-b)^{p-1}}{\rho+1-p}(a|a|_k^{\rho-p}-b|b|_k^{\rho-p}).
 	\end{split}
 	\end{align*}
 \end{proof}
 \noindent We now prove that a solution to \eqref{main} is bounded in $\Omega$.
 \begin{theorem}\label{bounded}
 	Let $f:\Omega\times\mathbb{R}\rightarrow\mathbb{R}$ be as defined in $(f_1)$, then for any weak solution $u\in X_0$, we have $u\in L^{\infty}(\Omega)$.
 \end{theorem}
 \begin{proof}
 	Let $1\leq q<p_s^*$ and let $u$ be any weak solution to the given problem in \eqref{main} and let $\tilde{\alpha}=\left(\frac{p_s^*}{p}\right)^{\frac{1}{p}}$. For every $\rho\geq p(p-1)$, $k>0$, the mapping $t\mapsto t|t|_k^{r-p}$ is Lipschitz in $\mathbb{R}$. Therefore, $u|u|_k^{\rho-p}\in X_0$. In general for any $t$ in $\mathbb{R}$ and $k>0$, we have defined $t_k=\text{sgn}(t)\min\{|t|,k\}$.  Note that, for a fixed solution of \eqref{main}, say $u$, we have that $1/\left(a+b\|u\|^p\right)$ is finite. We apply the embedding results due to Theorem \ref{embres}, the previous lemma \ref{ineq1}, test with the test function $u|u|_k^{\rho-p}$ and on using the growth condition of $f$ which is given in $(f_1)$ we get 
 	\begin{align}\label{crit_subcrit}
 	\begin{split}
 	\|u|u|_k^{\frac{\rho}{p}-1}\|_{p_s^*}^p\leq& C \|u|u|_k^{\frac{\rho}{p}-1}\|^p\\
 \leq	& C\frac{\rho^p}{\rho+1-p}\langle u,u|u|_k^{\rho-p} \rangle_{}\\
 \leq	& \rho^pC'\frac{1}{\left(a+b\|u\|^p\right)} \left(\lambda\int_{\Omega}g(x)|u|^{p-1}(|u||u|_k^{\rho-p})dx+\mu\int_{\Omega}h(x)|u|^{1-\gamma}(|u|_k^{\rho-p})dx\right.\\
 	&\left.+\int_{\Omega}|f(x,u)||u||u|_k^{\rho-p}dx\right)\\
 \leq	& C''\rho^{p}\int_{\Omega}\left(g(x)|u|^{p}|u|_k^{\rho-p}+h(x)|u|^{1-\gamma}|u|_k^{\rho-p}+|u||u|_k^{\rho-p}+|u|^q|u|_k^{\rho-p}\right)dx\\
 \leq & C''\rho^{p}\{\|g\|_{\infty}\int_{\Omega}|u|^{p}|u|_k^{\rho-p}dx+\|h\|_{\infty}\int_{\Omega}|u|^{1-\gamma}|u|_k^{\rho-p}dx\\
 &+\int_{\Omega}(|u||u|_k^{\rho-p}+|u|^q|u|_k^{\rho-p})dx\}\\
 \leq & C'''\rho^{p}\{\int_{\Omega}|u|^{p}|u|_k^{\rho-p}dx+\int_{\Omega}|u|^{1-\gamma}|u|_k^{\rho-p}dx+\int_{\Omega}(|u||u|_k^{\rho-p}+|u|^q|u|_k^{\rho-p})dx\}
 	\end{split}
 	\end{align}
 	for some $C'''>0$ independent of $\rho\geq p$ and $k>0$; but dependent on $g,h$ that have been considered in the problem. On applying the Fatou's lemma as $k\rightarrow\infty$ gives 
 	\begin{align}\label{eq0}
 	\begin{split}
 	\|u\|_{\tilde{\alpha}^p \rho}&\leq C'''\rho^{\frac{p}{\rho}}\left\{\int_{\Omega}(|u|^{\rho}+|u|^{\rho-(p-1)}+|u|^{\rho+q-p}\right.\\
 	&+\left.|u|^{\rho-p-\gamma+1})dx\right\}^{1/\rho}.
 	\end{split}
 	\end{align}
 	The idea is to try and develop an argument to guarantee that $u\in L^{p_1}(\Omega)$ for all $p_1\geq 1$.  Therefore define a recursive sequence $(\rho_n)$ by setting $\rho_0=p_s^*+p-q$, $\rho_{n+1}=\tilde{\alpha}^p \rho_n+q-p$. By the choice of $\rho_0$, we have $u\in L^{\rho+q-p}(\Omega)$. Therefore, the choice $\rho=\rho_0$ in \eqref{eq0} yields a finite right hand side and so $u\in L^{\tilde{\alpha}^p\rho}(\Omega)=L^{\rho_1+q-p}(\Omega)$.

 	 \noindent Repeating this argument and using the fact $\rho\mapsto \rho^{1/\rho}$ is bounded in $[2,\infty)$ for all $n$, we have $u\in L^{\tilde{\alpha}^p \rho_n}(\Omega)$. We further have 
 	\begin{align}\label{eq1}
 	\begin{split}
 	\|u\|_{\tilde{\alpha}^p \rho_n}&\leq G(n,\|u\|_{p_s^*}).
 	\end{split}
 	\end{align}
 	Arguments from Iannizzotto \cite{iannizzotto2015hs}, guarantees
 	\begin{align}\label{estimate1}
 	\begin{split}
 	\|u\|_{p_1}&\leq G(p_1,\|u\|_{p_s^*}), p_1\geq 1.
 	\end{split}
 	\end{align}
 	We now improve the estimate in (\ref{estimate1}) by making the function $G$ independent of $p_1$. Set $\tilde{\alpha}'=\frac{\tilde{\alpha}}{\tilde{\alpha}-1}$. Thus by (\ref{estimate1}) and H\"{o}lder's inequality we have
 	\begin{align*}
 	\begin{split}
 	\||u|+|u|^{p}+|u|^{q}\|_{\tilde{\alpha}'}&\leq G(\|u\|_{p_s^*})
 	\end{split}
 	\end{align*} 
 	Therefore for $\rho\geq p(p-1)$ we have 
 	\begin{align*}
 	\begin{split}
 	&\||u|^{\rho-(p-1)}+|u|^{\rho-p+p}+|u|^{r+q-p}+|u|^{\rho-p-\gamma+1}\|_{\tilde{\alpha}'}\\
 	&\leq \||u|+|u|^{p}+|u|^{q}+|u|^{1-\gamma}\|_{\alpha'}\||u|^{\rho-p}\|_{\tilde{\alpha}}\\
 	&\leq G(\|u\|_{p_s^*})\|u\|_{\tilde{\alpha}(\rho-p)}^{\rho-p}\\
 	&\leq G(\|u\|_{p_s^*})\|u\|_{\tilde{\alpha}^{p-1}(\rho-p)}^{\rho-p}\\
 	&\leq G(\|u\|_{p_s^*})|\Omega|^{\frac{1}{\tilde{\alpha}^{p-1}\rho}}\|u\|_{\tilde{\alpha}^{p-1}\rho}^{\rho-p}
 	\end{split}
 	\end{align*}
 	We note that $t\mapsto |\Omega|^{p/(\tilde{\alpha}^{p-1} t)}$ is a bounded map in $[p,\infty)$ and hence
 	\begin{align}\label{estimate3}
 	\begin{split}
 	&\||u|^{\rho-(p-1)}+|u|^{\rho-p+p}+|u|^{\rho+q-p}+|u|^{\rho-p-\gamma+1}\|_{\tilde{\alpha}'}\\
 	&\leq G(\|u\|_{p_s^*})\|u\|_{\tilde{\alpha}^{p-1}\rho}^{\rho-p}
 	\end{split}
 	\end{align}
 	For a sufficiently large $n$ we define $\rho=\tilde{\alpha}^{n-1}>>p$ and further set $v=\frac{u}{H(\|u\|_{p_s^*})^{1/p}}$. Using these choices in (\ref{eq0}) and the recursive formula we obtain we get
 	\begin{align}
 	\|u\|_{\tilde{\alpha}^{n}+p-1}^{\tilde{\alpha}^{n-1}}&\leq G(\|u\|_{p_s^*})\|u\|_{\tilde{\alpha}^{n-p+2}}^{\tilde{\alpha}^{n-1}-p}.
 	\end{align}
 	On using the definition of $v$ and iterating we get,
 	\begin{align*}
 	\begin{split}
 	\|v\|_{\tilde{\alpha}^{n+p-1}}&\leq \|v\|_{\tilde{\alpha}^{n+p-2}}^{1-p\tilde{\alpha}^{1-n}}\\
 	&\leq \|v\|_{\tilde{\alpha}^{n+p-3}}^{(1-p\tilde{\alpha}^{1-n})(1-(p-1)\tilde{\alpha}^{2-n})}\\ 
 	&\cdots\\
 	&\leq  \|v\|_{\tilde{\alpha}^{p}}^{\prod_{i=1}^{n-1}[1-p\tilde{\alpha}^{i-n}]}    
 	\end{split}
 	\end{align*}
 	It is easy to see that the product $\prod_{i=1}^{n-1}[1-p\tilde{\alpha}^{i-n}]$ is bounded in $\mathbb{R}$ and hence for all $n$ we have 
 	\begin{align*}
 	\begin{split}
 	\|v\|_{\tilde{\alpha}^{n+p-1}}&\leq  \|v\|_{\tilde{\alpha}^{p}}^{\prod_{i=1}^{n-1}[1-p\tilde{\alpha}^{i-n}]}<\infty.
 	\end{split}
 	\end{align*}
 	Reverting back to $u$ and recalling the fact that $\tilde{\alpha}^{n-1}\rightarrow\infty$ as $n\rightarrow\infty$, we find that there exists $H\in C(\mathbb{R}^+)$ such that $\|u\|_{p_1}\leq H(\|u\|_{p_s^*})$ for all $p_1 \geq 1$. The function $H$ here has been obtained from the function $G$ which was shown previously. Therefore, we have $\|u\|_{\infty}< \infty$.

 \end{proof}
 
 \begin{lemma}[Weak Comparison Principle]\label{weak comparison}
 	Let $u, v\in X_0$. Suppose, $(a+b\|v\|^p)\mathfrak{L}_p^sv-h(x)\frac{\mu}{v^{\gamma}}\geq(a+b\|u\|^p)\mathfrak{L}_p^su-h(x)\frac{\mu}{u^{\gamma}}$ weakly with $v=u=0$ in $\mathbb{R}^N\setminus\Omega$.
 	Then $v\geq u$ in $\mathbb{R}^N.$
 \end{lemma}
 \begin{proof}
 	Since, $(a+b\|v\|^p)\mathfrak{L}_p^sv-h(x)\frac{\mu}{v^{\gamma}}\geq(a+b\|u\|^p)\mathfrak{L}_p^su-h(x)\frac{\mu}{u^{\gamma}}$ weakly with $u=v=0$ in $\mathbb{R}^N\setminus\Omega$, we have
 	{\small\begin{align}\label{compprinci}
 		\langle(a+b\|v\|^p)\mathfrak{L}_p^sv,\phi\rangle-\int_{\Omega}h(x)\frac{\mu\phi}{v^{\gamma}}dx&\geq\langle(a+b\|u\|^p)\mathfrak{L}_p^su,\phi\rangle-\int_{\Omega}h(x)\frac{\mu\phi}{u^{\gamma}}dx
 		\end{align}}
 	$\forall{\phi\geq 0\in X_0}$. \\Suppose $S=\{x\in\Omega:u(x)>v(x)\}$ is a set of non-zero measure. Over this set $S$, we have 
 	{\small\begin{align}\label{compprinci1}
 		&(a+b\|v\|^p)\mathfrak{L}_p^sv-(a+b\|u\|^p)\mathfrak{L}_p^su\geq\mu h(x)\left(\frac{1}{v^{\gamma}}-\frac{1}{u^{\gamma}}\right)\geq 0.
 		\end{align}}
\noindent Define $\mathfrak{m}(t)=(a+bt^p)\geq a>0$ for $t\geq 0$ and $$\mathfrak{M}(t)=\int_{0}^{t}\mathfrak{m}(t)dt.$$ We will now show that the operator $\mathfrak{M}(\cdot)\mathfrak{L}_p^s(\cdot)$ is a {\it monotone} operator. By the Cauchy-Schwartz inequality we have 
\begin{eqnarray}\label{csineq}|(u(x)-u(y))(v(x)-v(y))|&=& |u(x)-u(y)||v(x)-v(y)|\nonumber\\
& \leq &\frac{|u(x)-u(y)|^2+|v(x)-v(y)|^2}{2}.\end{eqnarray}
Consider $I_1=\langle \mathfrak{m}(u)\mathfrak{L}_p^su,u\rangle-\langle \mathfrak{m}(u)\mathfrak{L}_p^su,v\rangle-\langle \mathfrak{m}(v)\mathfrak{L}_p^sv,u\rangle+\langle \mathfrak{m}(v)\mathfrak{L}_p^sv,v\rangle$ and let $|u(x)-u(y)| \geq|v(x)-v(y)|$.
Therefore using \eqref{csineq} we get 
\begin{eqnarray}
I_1&=&p \mathfrak{m}(\|u\|^p)\left(\int_{Q}|u(x)-u(y)|^{p-2}\{|u(x)-u(y)|^2\right.\nonumber\\
& &\left.-(u(x)-u(y))(v(x)-v(y))\}dxdy\right)\nonumber\\
& &+p \mathfrak{m}(\|v\|^p)\left(\int_{Q}|v(x)-v(y)|^{p-2}\{|v(x)-v(y)|^2\right.\nonumber\\
& &\left.-(u(x)-u(y))(v(x)-v(y))\}dxdy\right)\nonumber
\end{eqnarray}
\begin{eqnarray}\label{first_ineq}
&\geq&\frac{p}{2}  \mathfrak{m}(\|u\|^p)\left(\int_{Q}|u(x)-u(y)|^{p-2}\{|u(x)-u(y)|^2\right.\nonumber\\
& &\left.-|v(x)-v(y)|^2\}dxdy\right)\nonumber\\
& &+\frac{p}{2}  \mathfrak{m}(\|v\|^p)\left(\int_{Q}|v(x)-v(y)|^{p-2}\{|v(x)-v(y)|^2\right.\nonumber\\
& &\left.-|u(x)-u(y)|^2\}dxdy\right)\nonumber\\
&\geq&\frac{p}{2}\mathfrak{m}(\|u\|^p)\left(\int_{Q}(|u(x)-u(y)|^{p-2}-|v(x)-v(y)|^{p-2})(|u(x)-u(y)|^2-|v(x)-v(y)|^2)dx\right).\nonumber\\
&\geq&\frac{p}{2}a\left(\int_{Q}(|u(x)-u(y)|^{p-2}-|v(x)-v(y)|^{p-2})(|u(x)-u(y)|^2-|v(x)-v(y)|^2)dx\right).
\end{eqnarray}
When $|u(x)-u(y)| \leq|v(x)-v(y)|$, we interchange the roles of $u$, $v$ to get 
\begin{eqnarray}\label{second_ineq}
I_1&\geq&pa\left(\int_{Q}(|u(x)-u(y)|^{p-2}-|v(x)-v(y)|^{p-2})(|u(x)-u(y)|^2-|v(x)-v(y)|^2)dx\right).\nonumber\\
\end{eqnarray}
Thus
\begin{eqnarray}
\langle \mathfrak{m}(u)\mathfrak{L}_p^su-\mathfrak{m}(v)\mathfrak{L}_p^sv,u-v\rangle&=&I_1\geq 0.
\end{eqnarray}
{Thus $\mathfrak{m}(\cdot)\mathfrak{L}_p^s(\cdot)$ is a monotone operator. This monotonicity is sufficient for our work.\\
Coming back to} \eqref{compprinci1}, {by the monotonicity of $\mathfrak{m}(\cdot)\mathfrak{L}_p^s(\cdot)$ thus proved implies that $v\geq u$ in $S$. Therefore $u=v$ in $S$ and hence $u\geq v$ a.e. in $\Omega$.}
\end{proof}

\subsection{$C^1$ versus $W^{s,p}$ local minimizers of the energy} 
This section is devoted towards discussing {\it `$C^1$ versus $W^{s,p}$'} analysis of a solution to \eqref{main} for a particular class of Kernel $K(x)=|x|^{-N-sp}$. Some motivation has been drawn from the works of \cite{saoudi_add1, saoudi_add2, ghosh_jmp}. Let us begin with some well-known results and prove a few lemmas towards which a geometrical property of a general bounded domain $\Omega$ with $C^{1, 1}$ boundary is stated and is as follows.
\begin{lemma}[Lemma 3.5, Iannizzotto \cite{iannizzotto2014global}]\label{geo}
	{Let $\Omega\subset\mathbb{R}^N$ be a bounded domain with a $C^{1, 1}$ boundary $\partial\Omega$. Then, there exist $\rho>0$ such that for all $x_0\in\partial\Omega$ there exist $x_1, x_2\in\mathbb{R}^N$ on the normal line to $\partial\Omega$ at $x_0$, with the following properties}
	\begin{enumerate}[label=(\roman*)]
		\item $B_{r}(x_1)\subset\Omega$, $B_{r}(x_2)\subset\Omega^c$;
		\item $\bar{B}_{r}(x_1)\cap\bar{B}_{r}(x_2)=\{x_0\}$;
		\item $d(x)=|x-x_0|$ for all $x\in[x_0, x_1]$.
	\end{enumerate}
\end{lemma}
\noindent Using the Lemma \ref{geo}, we generalize two of the results from Iannizzotto \cite{iannizzotto2014global}. Before that, we set $\forall~R>0$, $x_0\in\mathbb{R}^N$
\begin{align}
\begin{split}
Q(u; x_0, R)&=\|u\|_{L^{\infty}(B_R(x_0))}+ \text{Tail}(u; x_0, R)\\Q(u, R)&=Q(u; 0, R)\\d(x,\partial\Omega)&=\inf_{y\in\Omega}\{d(x,y)\}.
\end{split}
\end{align}
We now state the following Lemmas from \cite{ghosh_jmp} which is applicable to the present work as well. The proofs follow verbatim of lemma $5.4, 5.6, 5.8$ of \cite{ghosh_jmp}.
\begin{lemma}\label{ianbdd} For any $r>0$, there exists $C''>0$ such that
	$|(a+b\|u\|^p)\mathfrak{L}_p^su|\leq C''$ in $B_r(x)$, where $u$ is a weak solution to the problem \eqref{main}.
\end{lemma}
\begin{lemma}\label{calpha1}
There exists $0<\delta\leq s$ such that any weak solution $u$ to the problem \eqref{main} we have $[u/d^s]_{C^{\delta}(\overline{\Omega})}\leq K$.
\end{lemma}
\begin{remark}
	From \cite{finebdry} we can say that $u\in C^1(\bar{\Omega})$.	
\end{remark}
\begin{lemma}\label{calpha2}
	There exists $0<\delta\leq s$ such that for any weak solution $u$ of the problem \eqref{main} we have $[Du]_{C^{\delta}(\overline{\Omega})}\leq C$.
\end{lemma}
\noindent We will now prove the Theorem \ref{regularity holder}. The main tool to prove this result requires an application of the Lagrange multiplier rule which is given in the form of Theorem $3.1$ from \cite{lagrange_1}.
\begin{theorem}\label{lagrangeSuff}
	Let $L$ and $J$ be real $C^{1}$ functionals on a real Banach space say $X$. If $z_0\in X$ satisfies the following problem:
	$$\text{minimizing}~L(z)~\text{under the constratint}~J(z)=0.$$
	Then there exists $\Lambda\in\mathbb{R}$ such that $L'(z_0)=\Lambda J'(z_0)$.
\end{theorem}
\noindent For a more generalized version of the result, one may refer to \cite{lagrange_2} and the references therein.\\
{\bf Proof of Theorem \ref{regularity holder}}:
Let $\Omega'\Subset\Omega$. We will work with the functional $\bar{I}$ since, a critical point of $\bar{I}$ is also a critical point of $I$ whenever $u>\underline{u}_{\mu}$. We will only consider the subcritical case i.e. when $q<p_s^*-1$.
We prove by contradiction, i.e. suppose $u_0$ is not a local minimizer. Let $r\in (q,p_s^*-1)$ and define
\begin{align}\label{aux1}
\begin{split}
J(w)&=\frac{1}{r+1}\int_{\Omega'}|w-u_0|^{r+1}dx, (w\in W^{s,p}(\Omega')).
\end{split}
\end{align}
{\bf Case i}:~{Let $J(v_{\epsilon})<\epsilon$}.\\
Define $S_{\epsilon}=\{v\in W_0^{s,p}(\Omega):0\leq J(v)\leq \epsilon\}$. Consider the problem $I_{\epsilon}=\underset{v\in S_{\epsilon}}{\min}\{\bar{I}(v)\}$. The infimum exists since the set $S_{\epsilon}$ is bounded and the functional $\bar{I}$ is $C^1$. Furthermore, $\bar{I}$ is also weakly lower semicontinuous and $S_{\epsilon}$ is closed, convex. Thus $I_{\epsilon}$ is actually attained, at say $v_{\epsilon}\in S_{\epsilon}$, and $I_{\epsilon}=\bar{I}(v_{\epsilon})<\bar{I}(u_0)$.\\
{\it Claim}:~We now show that $\exists \eta>0$ such that $v_{\epsilon}\geq \eta \phi_1$, where $\phi_1$ is the eigenvector corresponding to the {\it first}  eigenvalue, say  $\tilde{\lambda}_1$, of the operator $(a+b\|\cdot\|^p)\mathfrak{L}_p^s(\cdot)$.\\
{\it Proof}:~ We begin by observing that the existence of $\phi_1$ can be proved to exist from the Lemma \ref{existence_positive_soln}. Define $v_{\eta}=(\eta\phi_1-v_{\epsilon})^{+}$. We prove the claim by contradiction, i.e. $\forall\eta>0$ let $|\Omega_{\eta}|=|supp\{(\eta\phi_1-v_{\epsilon})^{+}\}|>0$.  For $0<t<1$, define $\xi(t)=\bar{I}(v_{\epsilon}+tv_{\eta})$. Thus 
{\small\begin{align}
	\begin{split}
	\xi'(t)&=\langle\bar{I}'(v_{\epsilon}+tv_{\eta}),v_{\eta}\rangle\\
	&=\langle(a+b\|u\|^p)\mathfrak{L}_p^s(v_{\epsilon}+tv_{\eta})-\lambda g(x) (v_{\epsilon}+tv_{\eta})^{p-1}-\mu h(x)(v_{\epsilon}+tv_{\eta})^{-\gamma}-f(x,v_{\epsilon}+tv_{\eta}),v_{\eta}\rangle.
	\end{split}
	\end{align}}
Similarly,
\begin{align}\label{aux_1}
\begin{split}
\xi'(1)=&\langle\bar{I}'(v_{\epsilon}+v_{\eta}),v_{\eta}\rangle\\
=&\langle\bar{I}'(\eta\phi_1),v_{\eta}\rangle\\
=&\langle(a+b\|\eta\phi_1\|^p)\mathfrak{L}_p^s(\eta\phi_1)-\lambda g(x) (\eta\phi_1)^{p-1}-\mu h(x)(\eta\phi_1)^{-\gamma}\\
&-f(x,\eta\phi_1),v_{\eta}\rangle<0
\end{split}
\end{align}
for sufficiently small $\eta>0$. Moreover,
\begin{align}\label{aux_2}
\begin{split}
-\xi'(1)+\xi'(t)=&\langle(a+b\|u\|^p)\mathfrak{L}_p^s(v_{\epsilon}+tv_{\eta})-(a+b\|u\|^p)\mathfrak{L}_p^s(v_{\epsilon}+v_{\eta})\\
&+\lambda g(x)((v_{\epsilon}+v_{\eta})^{p-1}-(v_{\epsilon}+tv_{\eta})^{p-1})\\
&+\mu h(x)((v_{\epsilon}+v_{\eta})^{-\gamma}-(v_{\epsilon}+tv_{\eta})^{-\gamma})\\
&+(f(x,v_{\epsilon}+v_{\eta})-f(v_{\epsilon}+tv_{\eta}),v_{\eta}\rangle\leq 0
\end{split}
\end{align}
since $\lambda g(x)t^{p-1}+\mu h(x)t^{-\gamma}+f(.,t)$ is a uniformly nonincreasing function with respect to, sufficiently small, $t>0$ for $x\in\Omega$. From the monotonicity of $(a+b\|\cdot\|^p)\mathfrak{L}_p^s(\cdot)$ (refer Theorem \ref{weak comparison}) we have that, for sufficiently small $\eta>0$, $0\leq \xi'(1)-\xi'(t)$. From the Taylor series expansion and the fact that $J(v_{\epsilon})<\epsilon$, $\exists~ 0<\theta<1$ such that 
\begin{align}
\begin{split}
0&\leq \bar{I}(v_{\epsilon}+v_{\eta})-\bar{I}(v_{\epsilon})\\
&=\langle\bar{I}'(v_{\epsilon}+\theta v_{\eta}),v_{\eta}\rangle\\
&=\xi'(\theta).
\end{split}
\end{align}  
Thus for $t=\theta$ we get $\xi'(\theta)\geq 0$ which is a contradiction to $\xi'(\theta)\leq\xi'(1)<0$ as obtained above in \eqref{aux_1} and \eqref{aux_2}. Thus $v_{\epsilon}\geq \eta\phi_1$ for some $\eta>0$.\\
In fact, from the Lemmas \ref{calpha1} and \ref{calpha2} we have $ \underset{\epsilon\in(0,1]}{\sup}\{\|v_{\epsilon}\|_{C^{1,\delta}(\bar{\Omega})}\}\leq C$. By the compact embedding $C^{1,\delta}(\bar{\Omega})\hookrightarrow C^{1,\kappa}\overline{\Omega})$, for any $\kappa<\delta$,  we have $v_{\epsilon}\rightarrow u_0$ which contradicts the assumption made. Hence $\bar{I}$ attains its minimum at $u_0$.\\
{\bf Case ii}:~{$J(v_{\epsilon})=\epsilon$.}\\
Let $v_{\eta}=(\eta\phi_1-v_{\epsilon})^+$ and $\xi(t)=\bar{I}(v_{\epsilon}+tv_{\eta})$. Then by arguments as in Case i, we have that $\xi$ is decreasing. This implies that $\bar{I}(v_{\epsilon})>\bar{I}(v_{\epsilon}+tv_{\eta})$. Since the functionals $\bar{I}$, $J$ are $C^1$, hence
in this case from the Lagrange multiplier rule (refer Theorem \ref{lagrangeSuff}) there exists $\Lambda_{\epsilon}\in\mathbb{R}$ such that  $\bar{I}'(v_{\epsilon})=\Lambda_{\epsilon}J'(v_{\epsilon})$. We will first show that $\Lambda_{\epsilon}\leq 0$. Suppose $\Lambda_{\epsilon}>0$, then $\exists~ \phi\in X_0$ such that 
\begin{align*}
\begin{split}
\langle\bar{I}'(v_{\epsilon}),\phi \rangle<0~ \text{and} ~\langle J'(v_{\epsilon}),\phi \rangle<0.
\end{split}
\end{align*}
Then for small $t>0$ we have 
\begin{align*}
\begin{split}
\bar{I}(v_{\epsilon}+t \phi)&<\bar{I}(v_{\epsilon})\\
J(v_{\epsilon}+t \phi)&<J(v_{\epsilon})=\epsilon
\end{split}
\end{align*}
which is a contradiction to $v_{\epsilon}$ being a minimizer of $\bar{I}$ in $S_{\epsilon}$.\\
We now consider the following two cases.\\
{\bf Case a}:~($\Lambda_{\epsilon}\in(-l,0)$ where $l>-\infty$).\\ 
Now consider the sequence of problems
\begin{align}
\begin{split}
(P_{\epsilon}):~(a+b\|u\|^p)\mathfrak{L}_p^su&=\lambda g(x)u^{p-1}+\mu h(x) u^{-\gamma}+f(x,u)+\Lambda_{\epsilon}|u-u_0|^{r-1}(u-u_0)
\end{split}
\end{align}
Observe that $u_0$ is a weak solution to ($P_{\epsilon}$). From the weak comparison principle (Theorem \ref{weak comparison}) we have $v_{\epsilon}\geq \eta\phi_1$ for some $\eta>0$ small enough, independent of $\epsilon$ since $\eta\phi_1$ is a strict subsolution to $(P_{\epsilon})$. Further, since $-l\leq \Lambda_{\epsilon}\leq 0$, there exist $M$, $c$ such that 
\begin{align}
\begin{split}
(a+b\|(v_{\epsilon}-1)^+\|^p)\mathfrak{L}_p^s(v_{\epsilon}-1)^{+}&\leq M+c((v_{\epsilon}-1)^{+})^r.
\end{split}
\end{align}
Using the Moser iteration technique as in Theorem \ref{bounded} we obtain $\|v_{\epsilon}\|_{\infty}\leq C'$. Therefore $\exists L>0$ such that $\eta\phi_1\leq v_{\epsilon}\leq L\phi_1$. By using the arguments previously used in {\it Case i}, we end up getting $|v_{\epsilon}|_{C^{1,\delta}(\bar{\Omega})}\leq C'$ for every $\epsilon>0$. The conclusion follows as in the previous case of $J(v_{\epsilon})<\epsilon$. Hence $\bar{I}$ attains its minimum at $u_0$.\\
{\bf Case b}:~$\underset{\epsilon>0}{\inf}\{\Lambda_{\epsilon}\}=-\infty$\\
Let us assume $\Lambda_{\epsilon}\leq -1$. As above, we can similarly obtain $v_{\epsilon}\geq \eta\phi_1$ for $\eta>0$ small enough and independent of $\epsilon$. Further, there exists a constant $M>0$ such that $\lambda g(x)t^{p-1}+\mu h(x)t^{-\gamma}+f(x,t)+\tau|t-u_0(x)|^{r-1}(t-u_0(x))<0$, $\forall (\tau,x,t)\in(-\infty,-1]\times\Omega\times(M,\infty)$.\\
From the weak comparison principle on $(a+b\|\cdot\|^p)\mathfrak{L}_p^s(\cdot)$, we get $v_{\epsilon}\leq M$ for $\epsilon>0$ sufficiently small. Since $u_0$ is a local $C^1$ -  minimizer, $u_0$ is a weak solution to \eqref{main} and hence
\begin{align}
\begin{split}
\langle (a+b\|u_0\|^p)\mathfrak{L}_p^su_0,\phi \rangle&=\lambda\int_{\Omega}g(x)u_0^{p-1}\phi dx+\mu\int_{\Omega}h(x)u_0^{-\gamma}\phi dx+\int_{\Omega}f(x,u_0)\phi dx
\end{split}
\end{align}
$\forall\phi\in C_c^{\infty}(\Omega)$. Also, $u_0$ satisfies 
\begin{align}\label{s1}
\begin{split}
\langle (a+b\|u_0\|^p)\mathfrak{L}_p^su_0,w \rangle&=\lambda\int_{\Omega}g(x)u_0^{p-1}w+\mu\int_{\Omega}h(x)u_0^{-\gamma}w dx+\int_{\Omega}f(x,u_0)w dx.
\end{split}
\end{align}
Similarly,
\begin{align}\label{s2}
\begin{split}
\langle (a+b\|v_{\epsilon}\|^p)\mathfrak{L}_p^sv_{\epsilon},w \rangle&=\lambda\int_{\Omega}g(x)v_{\epsilon}^{p-1}w+\mu\int_{\Omega}h(x)v_{\epsilon}^{-\gamma}w dx+\int_{\Omega}f(x,v_{\epsilon})w dx.
\end{split}
\end{align}
On subtracting \eqref{s1} from \eqref{s2} and testing with $w=|v_{\epsilon}-u_0|^{\beta-1}(v_{\epsilon}-u_0)$, where $\beta\geq 1$, we obtain
\begin{align}
\begin{split}
0=& \beta\langle (a+b\|u_{\epsilon}\|^p)\mathfrak{L}_p^sv_{\epsilon}-(a+b\|u_0\|^p)\mathfrak{L}_p^su_0,|v_{\epsilon}-u_0|^{\beta-1}(v_{\epsilon}-u_0) \rangle\\
&-\lambda\int_{\Omega}g(x)(v_{\epsilon}^{p-1}-u_0^{p-1})|v_{\epsilon}-u_0|^{\beta-1}(v_{\epsilon}-u_0)dx\\
&-\mu\int_{\Omega}h(x)(v_{\epsilon}^{-\gamma}-u_0^{-\gamma})|v_{\epsilon}-u_0|^{\beta-1}(v_{\epsilon}-u_0)dx\\
=&\int_{\Omega}(f(x,v_{\epsilon})-f(x,u_0))|v_{\epsilon}-u_0|^{\beta-1}(v_{\epsilon}-u_0)dx\\
&+\Lambda_{\epsilon}\int_{\Omega}|v_{\epsilon}-u_0|^{\beta+r}dx.
\end{split}
\end{align}
By the H\"{o}lder's inequality and the bounds of $v_{\epsilon}$, $u_0$ we obtain
\begin{align}
\begin{split}
-\Lambda_{\epsilon}\|v_{\epsilon}-u_0\|_{\beta+r}^{r}&\leq C|\Omega|^{\frac{r}{\beta+r}}.
\end{split}
\end{align}
Here $C$ is independent of $\epsilon$ and $\beta$. On passing the limit $\beta\rightarrow\infty$ we get $-\Lambda_{\epsilon}\|v_{\epsilon}-u_0\|_{\infty}\leq C$. Working on similar lines we end up getting $v_{\epsilon}$ is bounded in $C^{1,\delta}(\bar{\Omega})$ independent of $\epsilon$ and the conclusion follows.
\section{Appendix}
The appendix will address a few results that have been used in this article. Lemma \ref{existence_positive_soln} guarantees the existence of a positive solution to \eqref{auxprob}, Lemma \ref{MP_geometry} will establish that the functional $\bar{I}$ verifies the mountain pass geometry, whereas Lemma \ref{u_greater_u_lambda} guarantees that a solution to \eqref{main} is greater than or equal to the solution to \eqref{auxprob}.
\begin{remark}\label{u_greater_than_zero}
By saying ``$u>0$ in $\Omega$'' we will mean $\underset{V}{\text{ess}\inf} u>0$ for any compact set $V\subset\Omega$.
\end{remark}
\begin{lemma}
	\label{existence_positive_soln}
	Let $0<\gamma<1$, $\lambda,\mu>0$. Then the following problem 
	\begin{eqnarray}\label{auxprob_appendix}
	\left(a+b\int_{\mathbb{R}^N}|u(x)-u(y)|^{p}K(x-y)dxdy\right)\mathfrak{L}_p^su&=&\mu h(x)u^{-\gamma},~\text{in}~\Omega\nonumber\\
	u&>&0,~\text{in}~\Omega\nonumber\\
	u&=&0,~\text{in}~\mathbb{R}^N\setminus\Omega\nonumber\\
	\end{eqnarray}
	has a unique weak solution in $X_0$. This solution is denoted by $\underline{u}_{\mu}$, satisfies $\underline{u}_{\mu}\geq \epsilon_{\mu} v_0$ a.e. in $\Omega$, where $\epsilon_{\mu}>0$ is a constant.
\end{lemma}
\begin{proof}
We follow the proof in \cite{giaco_1}. Firstly, we note that an energy functional on $X_0$ formally corresponding to \eqref{auxprob_appendix} can be defined as follows.
\begin{align}
\label{ef_aux}
E(u)&=\frac{a}{p}\|u\|^p+\frac{b}{2p}\|u\|^{2p}-\frac{\mu}{1-\gamma}\int_{\Omega}h(x)(u^+)^{1-\gamma}dx
\end{align}
for $u\in X_0$. By the Poincar\'{e} inequality, this functional is coercive and continuous on $X_0$. It follows that $E$ possesses a global minimizer $u_0\in X_0$. Clearly, $u_0\neq 0$ since $E(0)=0>E(\epsilon v_0)$ for sufficiently small $\epsilon$ and some $v_0>0$ in $\Omega$.\\
Secondly, we have the decomposition $u=u^+-u^-$. Thus if $u_0$ is a global minimizer for $E$, then so is $|u_0|$, by $E(|u_0|)\leq E(u_0)$. Clearly enough, the equality holds iff $u_0^-=0$ a.e. in $\Omega$. In other words we need to have $u_0\geq 0$, i.e. $u_0\in X_0$ where 
$$X_0^+=\{u\in X_0:u\geq 0~\text{a.e. in}~\Omega\}$$
is the positive cone in $X_0$.\\
Third, we will show that $u_0\geq \epsilon v_0>0$ holds a.e. in $\Omega$ for small enough $\epsilon$. Observe that, 
\begin{align}\label{neg_der}
\begin{split}
E'(tv_0)|_{t=\epsilon}=&a\epsilon^{p-1}\|v_0\|^p+b\epsilon^{2p-1}\|v_0\|^{2p}-\mu\epsilon^{-\gamma}\int_{\Omega}h(x)v_0^{1-\gamma}dx<0
\end{split}
\end{align}
whenever $0<\epsilon\leq \epsilon_{\mu}$ for some sufficiently small $\epsilon_{\mu}$. We now show that $u_0\geq \epsilon_{\mu}v_0$. On the contrary, suppose $w=(\epsilon_{\mu}v_0-u_0)^+$ does not vanish identically in $\Omega$. Denote $$\Omega^+=\{x\in\Omega:w(x)>0\}.$$
We will analyse the function $\zeta(t)=E(u_0+tw)$ of $t\geq 0$. This function is convex owing to its definition over $X_0^+$ being convex. Further $\zeta'(t)=\langle E'(u_0+tw),w \rangle$ is nonnegative and nondecreasing for $t>0$. Consequently for $0<t<1$ we have
\begin{align}\label{ineq_appendix}
\begin{split}
0\leq \zeta'(1)-\zeta'(t)&=\langle E'(u_0+w)-E'(u_0+tw),w\rangle\\
&=\int_{\Omega^+}E'(u_0+w)dx-\zeta'(t)\\
&<0
\end{split}
\end{align}
by inequality \eqref{neg_der} and $\zeta'(t)\geq 0$ with $\zeta'(t)$ being nondecreasing for every $t>0$, which is a contradiction. Therefore $w=0$ in $\Omega$ and hence $u_0\geq \epsilon_{\mu}v_0$ a.e. in $\Omega$.\\
Finally, the functional $E$ being strictly convex on $X_0^+$, we conclude that $u_0$ is the only critical point of $E$ in $X_0^+$ with the property $\underset{V}{\text{ess}\inf}u_0>0$ for any compact subset $V\subset\Omega$. Therefore we choose $\underline{u}_{\mu}=u_0$ in the cutoff functional.
\end{proof}
\begin{remark}\label{obs_pos}  We now perform an apriori analysis on a solution (if it exists). Suppose $u$ is a solution to \eqref{main}, then we observe the following
	\begin{enumerate}
		\item $I(u)=I(|u|)$. This implies that $u^-=0$ a.e. in $\Omega$.
		\item In fact a solution to \eqref{main} can be considered to be positive, i.e. $u>0$ a.e. in $\Omega$ due to the presence of the singular term.
	\end{enumerate}
	Thus without loss of generality, we assume that the solution is positive. 
\end{remark}
Precisely, we now have the following result.
\begin{lemma}[Apriori analysis]\label{u_greater_u_lambda}
	Fix a $\mu\in(0,\mu_0)$. Then a solution of \eqref{main}, say $u>0$, is such that $u\geq \underline{u}_{\lambda}$ a.e. in $\Omega$.
\end{lemma}
\begin{proof}
	Fix $\mu\in(0,\mu_0)$ and let $u\in X_0$ be a positive solution to \eqref{main} and $\underline{u}_{\lambda}>0$ be a solution to \eqref{auxprob_appendix}. We will show that $u\geq \underline{u}_{\lambda}$ a.e. in $\Omega$. Thus, we let $\underline{\Omega}=\{x\in\Omega:u(x)<\underline{u}_{\lambda}(x)\}$ and from the equation satisfied by $u$, $\underline{u}_{\lambda}$, we have 
	\begin{align}\label{comp_1}
	0\leq&\langle(a+b\|\underline{u}_{\lambda}\|^p)\mathfrak{L}_p^s\underline{u}_{\lambda}-(a+b\|u\|^p)\mathfrak{L}_p^su,\underline{u}_{\lambda}-u\rangle_{\underline{\Omega}}	+\lambda\int_{\underline{\Omega}}g(x)u^{p-1}(\underline{u}_{\lambda}-u)dx\nonumber\\
	\leq&\mu\int_{\underline{\Omega}}h(x)(\underline{u}_{\lambda}^{-\gamma}-u^{-\gamma})(\underline{u}_{\lambda}-u)dx\leq 0.
	\end{align}
	Further we have 
	\begin{align}\label{comp_2}
	\langle(a+b\|\underline{u}_{\lambda}\|^p)\mathfrak{L}_p^s\underline{u}_{\lambda}-(a+b\|u\|^p)\mathfrak{L}_p^su,\underline{u}_{\lambda}-u\rangle_{\underline{\Omega}}&\geq 0. 
	\end{align}
	Hence, from \eqref{comp_1} and \eqref{comp_2}, we obtain $u\geq\underline{u}_{\lambda}$ a.e. in $\Omega^c$.
\end{proof}
\begin{lemma}\label{MP_geometry}
	The redefined functional $\bar{I}$ given in \eqref{modi_func} verifies the mountain pass geometry for $\mu\in(0,\mu_0)$ with $\mu_0<\infty$.
\end{lemma}
\begin{proof}
	By the Sobolev embedding we obtain$$\bar{I}(u)\geq\frac{a}{p}\|u\|^p+\frac{b}{2p}\|u\|^{2p}-\frac{\lambda\|g\|_{\infty}C_1}{p}\|u\|^p-\frac{\mu\|h\|_{\infty}C_2}{1-\gamma}\|u\|^{1-\gamma}-\int_{\Omega}F(x,u)dx.$$
	where $C_1, C_2>0$ are uniform constants that are independent of the choice of $u$ and $F(x,t)=\int_{0}^{t}f(x,\omega)d\omega$. Now for a pair $(\mu,r)$, sufficiently small $\mu>0$ say $\mu_0$, we have that $\frac{a}{p}\|u\|^p+\frac{b}{2p}\|u\|^{2p}-\frac{\mu\|h\|_{\infty}C_2}{1-\gamma}\|u\|^{1-\gamma}>0$ for each $\mu\in(0,\mu_0)$ and $\|u\|=r$ sufficiently small. Define $a(r)=\frac{a}{p}r^p+\frac{b}{2p}r^{2p}-\frac{\mu\|h\|_{\infty}C_2}{1-\gamma}r^{1-\gamma}$. Therefore, to sum it up we have
	$$\bar{I}(u)\geq a(r)>0$$
	for any $\mu\in(0,\mu_0)$ and for every $u$ such that $\|u\|=r$. On the other hand, taking $u\in X_0$ and $t\geq 0$ we have $\bar{I}(tu)\rightarrow-\infty$ as $t\rightarrow\infty$. This verifies the second condition of the Mountain pass theorem.
\end{proof}

\section*{Conclusions}
Existence of infinitely many solutions to the problem in \eqref{main} has been proved. In addition, a weak comparison result has been proved. It has also been shown that the solutions are in $L^{\infty}(\Omega)$ when $K(x)=|x|^{-N-sp}$. Further for this particular kernel $K$, it has also been proved that the $C^1$ minimizers are the $W_0^{s,p}$ minimizers as well. Some future scope of work on this line would be to prove that the $C^1$ minimizers are the $W_0^{s,p}$ minimizers as well for a general Kernel satisfying \eqref{kernel}.

\section*{Acknowledgement}
The author thanks S. Ghosh for the numerous discussion sessions and the constructive criticisms on the article. Thanks are due to the anonymous reviewers for their constructive comments that led to the improvement of this manuscript. The author also dedicates this article to thousands of labourers and workers of India who, during this COVID19  pandemic, have lost their lives travelling on foot for thousands of kilometres to reach their respective homes.



\begin{thebibliography}{10}
	\bibitem{amb1}
	A. Ambrosetti, P.H. Rabinowitz,
	\newblock Dual variational methods in critical point theory and applications,
	\newblock {\em Journal of Functional Analysis}, 14, 349-381, 1973.
	
	\bibitem{saoudi_sugg_3}
	A. Daoues, A. Hammami, K. Saoudi,
	\newblock Multiple positive solutions for a nonlocal PDE with critical Sobolev-Hardy and singular nonlinearities via perturbation method,
	\newblock {\em Fract. Calc. Appl. Anal.}, 23 (3), 837-860, 2020.
	
	\bibitem{fisval}
	A. Fiscella, E. Valdinoci,
	\newblock A critical Kirchhoff type problem involving a nonlocal operator,
	\newblock {\em Nonlinear Analysis }, 
	94, 156-170 , 2014.
	
	\bibitem{fis1}
	A. Fiscella, P. Pucci,
	\newblock $p$-fractional Kirchhoff equations involving critical nonlinearities,
	\newblock {\em Nonlinear Analysis: Real World Applications}, 
	35, 350-378 , 2017.
	
	\bibitem{ghanmi2016nehari}
	A. Ghanmi  and K. Saoudi,
	\newblock The Nehari manifold for a singular elliptic equation involving the fractional Laplace operator,
	\newblock {\em Fractional Differential Calculus}, 6(2), 201-217, 2016.
	
	
	
	\bibitem{ian}
	A. Iannizzotto, M. Squassina, 
	\newblock $\frac{1}{2}$-Laplacian problems with exponential nonlinearity,
	\newblock {\em J. Math. Anal. Appl.}, 414(1), 372-385, 2014.
	
	\bibitem{iannizzotto2015hs}
	A. Iannizzotto, S. Mosconi and M. Squassina,
	\newblock $H^s$ versus $C^0$-weighted minimizers,
	\newblock {\em Nonlinear Differential Equations and Applications NoDEA},
	22(3), 477-497, 2015.
	
	\bibitem{iannizzotto2014global}
	A. Iannizzotto, S. Mosconi and M. Squassina,
	\newblock Global H\"{o}lder regularity for the fractional $p$-Laplacian,
	\newblock {\em Revista Matem\'{a}tica Iberoamericana}, 32(4), 1353-1392, 2016.
	
	\bibitem{finebdry}
	A. Iannizzotto., S. Mosconi and M. Squassina,
	\newblock Fine boundary regularity for the degenerate fractional $p$-Laplacian,
	\newblock {\em Journal Functional Analysis}, 279(8), 108659 1-54, 2020.
	
	\bibitem{soni1}
	A. Soni, D. Choudhuri,
	\newblock Existence of multiple solutions to an elliptic problem with measure data,
	\newblock {\em J Elliptic Parabol. Equ.}, 4(2), 369-388, 2018.
	
	
	\bibitem{barr}
	B. Barrios, E, Colorado, A. De Pablo, U. S\'{a}nchez,
	\newblock On some critical problems for the fractional Laplacian operator,
	\newblock {\em Journal of Differential Equations}, 252(11), 6133-6162, 2012.
	
	\bibitem{bin}
	B. Zhang, G. Molica Bisci and R. Servadei,
	\newblock  Superlinear nonlocal fractional problems with infinitely many solutions,
	\newblock {\em Nonlinearity}, 28(7), 2247-2264, 2015.
	
	\bibitem{alv}
	C.O. Alves, F.J.S.A. Corr$\hat{e}a$ and T.F. Ma,
	\newblock  Positive solutions for a quasilinear elliptic equation of Kirchhoff type,
	\newblock {\em Computers \& Mathematics with Applications}, 49(1), 85-93, 2005.
	
	\bibitem{col}
	F. Colasuonno, P. Pucci,
	\newblock Multiplicity of solutions for $p(x)$-polyharmonic elliptic Kirchhoff equations,
	\newblock {\em Nonlinear Analysis: Theory, Method and Applications}, 74(17), 5962-5974, 2011.
	
	
	\bibitem{carrier}
	G.F. Carrier,
	\newblock On the non-linear vibration problem of the elastic string,
	\newblock {\em Quart. Appl. Math.}, 3, 157-165, 1945.
	
	
	
	
	\bibitem{moli1}
	G. Molica Bisci and B.A. Pansera,
	\newblock Three weak solutions for nonlocal fractional equations,
	\newblock {\em Advanced Nonlinear Studies}, 14(3), 619-629, 2014.
	
	\bibitem{moli3}
	G. Molica Bisci,
	\newblock  Sequences of weak solutions for fractional equations,
	\newblock {\em Math. Res. Lett.}, 21(2), 241-253, 2014.
	
	
	\bibitem{moli2}
	G. Molica Bisci,
	\newblock Fractional equations with bounded primitive,
	\newblock {\em Applied Math. Lett.}, 27, 53-58, 2014.
	
	\bibitem{giaco_1}
	J. Giacomoni, I, Schindler, P. Tak\'{a}\v{c},
	\newblock Sobolev versus H\"{o}lder local minimizers and existence of multiple solutions for a singular quasilinear equation,
	\newblock{Ann. Scuola Norm. Sup. Pisa Cl. Sci. $(5)$},
	6(2), 117-158, 2007.
	
	\bibitem{saoudi_add1}
	J. Giacomoni, K. Saoudi,
	\newblock $W_0^{1,p}$ versus $C^1$ local minimizers for a singular and critical functional,
	\newblock {\em J. Math. Anal. and Appl.}, 363 (2), 697-710, 2010.
	
	
	
	\bibitem{lagrange_2}
	J. Zowe, S. Kurcyusz,
	\newblock Regularity and stability for the mathematical programming problem in Banach spaces,
	\newblock {\em Appl. Math. Optim.}, 5(1), 49-62, 1979.
	
	\bibitem{zuo2}
	J. Zuo, T. An, W. Liu,
	\newblock A variational inequality of Kirchhoff type in $\mathbb{R}^N$,
	\newblock {\em Journal of Inequalities and Appl.}, 329, 1-9, 2018.
	
	\bibitem{zuo3}
	J. Zuo, T. An, M. Li,
	\newblock Superlinear Kirchhoff-type problems of the fractional $p$-Laplacian without the $(AR)$ condition,
	\newblock {\em Boundary Value Problems}, 
	2018, p. 180, 2018.
	
	\bibitem{zuo1}
	J. Zuo, T. An, L. Yang, X. Ren,
	\newblock  The Nehari manifold for a fractional $p$-Kirchhoff system involving sign-changing weight function and concave-convex nonlinearities,
	\newblock {\em J. of Function Spaces}, Art. ID 7624373, 9 Pages, 2019.
	
	\bibitem{zuo4}
	J. Zuo, T. An, X. Li, Y. Ma,
	\newblock  A fractional $p$-Kirchhoff type problem involving a parameter,
	\newblock {\em Journal of Nonlinear Functional Analysis}, 2019, p. 32, 1-14, 2019.
	
	\bibitem{saoudi_add2}
	K. Saoudi,
	\newblock On $W^{s,p}$ vs. $C^1$ local minimizers for a critical functional related to
	fractional $p$-Laplacian,
	\newblock{\em Applicable Analysis}, 96(9), 1586-1595, 2017.
	
	\bibitem{saoudi_sugg_1}
	K. Saoudi,
	\newblock  $W^{1,N}$ versus $C^1$ local minimizer for a singular functional with Neumann boundary condition,
	\newblock {\em Bol. Soc. Paran. Mat. (3)}, 37 (1), 71-86, 2019.
	
	\bibitem{saoudisugg2}
	K. Saoudi,
	\newblock A critical fractional elliptic equation with singular nonlinearities, {\it Fract. Calc. Appl. Anal.}, 20(6), 1507-1530, 2017.
	
	\bibitem{ghosh_jmp}
	K. Saoudi, S. Ghosh, D. Choudhuri,
	\newblock Multiplicity and H\"{o}lder regularity of solutions
	for a nonlocal elliptic PDE involving singularity,
	\newblock {\em J. Math. Phys.}, 60,  101509 1-28, 2019.
	
	
	
	\bibitem{tang}
	K. Teng,
	\newblock Two nontrivial solutions for hemivariational inequalities driven by nonlocal elliptic operators,
	\newblock {\em  Nonlinear Analysis: Real World Applications}, 14(1), 867-874, 2013.
	
	\bibitem{jean}
	L. Jeanjean,
	\newblock On the existence of bounded Palais-Smale sequences and application to a Landesman-Lazer-type problem set on $\mathbb{R}^N$,
	\newblock {\em Proceedings of the Royal Society of Edingurgh, Section A: Mathematics}, 129(4), 787-809, 1999.
	
	\bibitem{wang1}
	L. Wang, K. Xie, B. Zhang,
	\newblock Existence and multiplicity of solutions for critical Kirchhoff-type $p$-Laplacian problems,
	\newblock {\em J. Math. Anal. Appl.}, 
	458(1), 361-378, 2018.
	
	\bibitem{cap1}
	M. Caponi, P. Pucci,
	\newblock Existence theorems for entire solutions of stationary Kirchhoff fractional $p$-Laplacian equations,
	\newblock {\em Annali di Matematica Pura ed Applicata}(4), 195(6), 2099-2129, 2016.
	
	\bibitem{lagrange_1}
	M. Khaled, M. Rhoudaf, H. Sabiki,
	\newblock Lagrange multiplier rule to a nonlinear eigenvalue problem in Musielak-Orlicz spaces,
	\newblock {\em Numerical Functional Analysis and Optimization}, 41(2), 134-157, 2020.
	
	\bibitem{saoudi_sugg_4} 
	M. Kratou,
	\newblock Ground state solutions of $p$-Laplacian singular Kirchhoff problem involving a Riemann-Liouville fractional derivative,
	\newblock {\em  Filomat}, 33 (7), 2073-2088, 2019.  
	
	
	\bibitem{pan1}
	N. Pan, B. Zhang, J. Cao,
	\newblock Degenerate Kirchhoff-type diffusion problems involving the fractional $p$-Laplacian,
	\newblock {\em Nonlinear Analysis: Real World Applications}, 
	37, 56-70, 2017.
	
	\bibitem{nya}
	N. Nyamoradi, L.I. Zaidan,
	\newblock Existence of solutions for degenerate Kirchhoff type problems with fractional $p$-Laplacian,
	\newblock {\em EJDE}, 
	2017, p. 115, 1-13, 2017.
	
	\bibitem{thin1}
	N. V. Thin,
	\newblock Nontrivial solutions of some fractional problems,
	\newblock {\em Nonlinear Analysis: Real World Applications}, 38, 146-170, 2017.
	
	\bibitem{mish}
	P.K. Mishra, K. Sreenadh,
	\newblock Fractional $p$-Kirchhoff system with sign-changing nonlinearities,
	\newblock {\em Revista de la Real Academia de Ciencias Exactas Fisicas y Naturales. Serie A Matematicas}, 111(1), 281-296, 2017.
	
	\bibitem{pucci}
	P. Pucci, S. Saldi,
	\newblock  Critical stationary Kirchhoff equations in $\mathbb{R}^N$ involving nonlocal operators,
	\newblock {\em Revista Mat. Iberoamericana}, 32(1), 1-22, 2016.
	
	\bibitem{adam}
	R.A. Adams, J.J. Fournier,
	\newblock Sobolev spaces,
	\newblock {\em Academic press, NY, USA}, second edition, 2003.
	
	\bibitem{serva2}
	R. Servadei , E. Valdinoci,
	\newblock  Mountain pass solutions for non-local elliptic operators,
	\newblock {\em J. Math. Anal. Appl.}, 389(2), 887-898, 2012.
	
	\bibitem{serva1}
	R. Servadei, E. Valdinoci,
	\newblock  Lewy-Stampacchia type estimates for variational inequalities driven by (non)local operators,
	\newblock {\em Revista Mat. Iberoamericana}, 29(3), 1091-1126, 2013.
	
	\bibitem{serva3}
	R. Servadei, E. Valdinoci,
	\newblock Variational methods for nonlocal operators of elliptic type,
	\newblock {\em Discrete and continuous dynamical systems - Series A}, 33(5), 2105-2137, 2013.
	
	\bibitem{ghosh1}
	S. Ghosh,
	\newblock An existence result for singular nonlocal fractional Kirchhoff-Schr\"{o}dinger-Poisson system,
	\newblock {\em arXiv:1909.13350} 
	
	\bibitem{mukherjee2016dirichlet}
	T. Mukherjee, K. Sreenadh,
	\newblock On Dirichlet problem for fractional $p$-Laplacian with singular
	non-linearity,
	\newblock {\em Advances in Nonlinear Analysis}, 8(1), 52-72, 2019.
	
	
	\bibitem{ming3}
	X. Mingqi, G. M. Bisci, G. Tian, B. Zhang,
	\newblock Infinitely many solutions for the stationary Kirchhoff problems involving the fractional $p$-Laplacian,
	\newblock {\em Nonlinearity}, 29(2), 357-374, 2016.
	
	\bibitem{ming1}
	X. Mingqi, V.D. R\u{a}dulescu, B. Zhang,
	\newblock Combined effects for fractional Schr\"{o}dinger-Kirchhoff systems with critical nonlinearities,
	\newblock {\em ESAIM: Control, Optimisation and Calculus of Variations}, 
	24(3), 1249-1273, 2018.
	
	\bibitem{ming2}
	X. Mingqi, V.D. R\u{a}dulescu, B. Zhang,
	\newblock Nonlocal Kirchhoff diffusion problems: local existence and blow-up of solutions,
	\newblock {\em Nonlinearity}, 
	31(7), 3228-3250, 2018.
	
	\bibitem{xiang1}
	X. Mingqi, B. Zhang and M. Ferrara,
	\newblock  Existence of solutions for Kirchhoff type problem involving the non-local fractional $p$-Laplacian,
	\newblock {\em J. Math. Anal. Appl.}, 424(2), 1021-1041, 2015.
	
	\bibitem{xiang5}
	X. Mingqi, B. Zhang,
	\newblock Degenerate Kirchhoff problems involving the fractional $p$-Laplacian without the $(AR)$ condition,
	\newblock {\em Complex variables and elliptic Equ.}, 60(9), 1277-1287, 2015.
	
	\bibitem{xiang3}
	X. Mingqi, B. Zhang, X. Guo,
	\newblock Infinitely many solutions for a fractional Kirchhoff type problem via fountain theorem,
	\newblock {\em Nonlinear Analysis: Theory, Methods and Applications}, 120, 299-313, 2015.
	
	\bibitem{xiang4}
	X. Mingqi, B. Zhang, V.D. R\u{a}dulescu,
	\newblock Existence of solutions for a bi-nonlocal fractional $p$-Kirchhoff type problem,
	\newblock {\em Computers \& Mathematics with Applications}, 71(1), 255-266, 2016.
	
	\bibitem{xiang2}
	X. Mingqi, B. Zhang, H. Qiu,
	\newblock Existence of solutions for a critical fractional Kirchhoff type problem in $\mathbb{R}^N$,
	\newblock {\em Science China Mathematics}, 60(9), 1647-1660, 2017.
	
	\bibitem{ren1}
	X. Ren, J. Zuo, Z. Qiao, L. Zhu,
	\newblock Infinitely many solutions for a superlinear fractional $p$-Kirchhoff-type problem without the $(AR)$ condition,
	\newblock {\em Advances in Mathematical Physics}, 1353961-1-10, 2019.	
	
\end{thebibliography}
\end{document}